\documentclass[a4paper]{article}

\usepackage{amsfonts}
\usepackage{amsthm}
\usepackage{amssymb}
\usepackage{hyperref}
\usepackage{enumitem}
\usepackage{pgf}
\usepackage{float}
\usepackage{mathtools}
\usepackage{graphicx}

\usepackage{imakeidx}
\usepackage{calc}
\usepackage{cancel}
\usepackage{bbm}
\usepackage{ bbold }
\usepackage{comment}
\usepackage{cite}
\usepackage{hhline}
\usepackage{multirow}
\usepackage[left=3cm, right=3cm]{geometry}

\newtheorem{theorem}{Theorem}[section]
\newtheorem{prop}[theorem]{Proposition}
\newtheorem{coro}[theorem]{Corollary}
\newtheorem{lemma}[theorem]{Lemma}
\newtheorem{example}[theorem]{Example}

\newtheorem{definition}[theorem]{Definition}
\newtheorem{remark}[theorem]{Remark}

\newcommand{\ECT}{\operatorname{ECT}}
\newcommand{\bb}{\mathbb}
\newcommand{\dd}{\operatorname{d}}
\newcommand{\dist}[2]{\operatorname{dist}^{\operatorname{ECT}}_{(#1, #2)}}
\newcommand{\Peri}{\operatorname{Peri}}

\newcommand{\la}{\langle}
\newcommand{\ra}{\rangle}
\newcommand{\dsupp}{d^{\operatorname{supp}}}
\newcommand{\dcent}{d^{\operatorname{cent}}}
\newcommand{\dbottle}{{d_i^{\operatorname{bottle}}}}
\newcommand{\normsupp}[1]{\|#1\|^{\operatorname{supp}}}
\newcommand{\normcent}[1]{\|#1\|^{\operatorname{cent}}}

\title{The Euler Characteristic Transform from a Convex Geometric Perspective}
\author{Jesús Gacías Franco \\ \\ Universidad de Zaragoza, \url{jgacias@unizar.es}}
\date{July 30, 2026}

\begin{document}
	\maketitle

	\begin{abstract}
		By examining the relationship between the support function of convex geometry and the Euler Characteristic Transform (ECT) of topological data analysis, we develop new tools and suggest variations on some common ECT pipelines. Specifically, we put forward new definitions of ECT-induced pseudodistances, which have the advantage of being invariant under common euclidean isometries and require no cutoff parameter to compare shapes with distinct Euler characteristic. These definitions rely on a generalization of the convex geometric concept of the \emph{Steiner point}, which we define in general as a distinguished point given by the ECT. We also show how convex geometry provides a path to recover interesting geometric information of a flat shape from its ECT, namely, its perimeter, for which we give an explicit formula. By building on these concepts and leveraging persistent homology, we define \emph{Steiner barcodes} as an isometry invariant feature of shapes, as well as homological variants of the support function and Steiner point. Finally, we put these constructions to the test in shape classification tasks, providing lightweight features for aligned and misaligned datasets.
	\end{abstract}
	
	\tableofcontents
	
	\section{Introduction}
\label{SectIntro}
A classical idea in convex geometry, and more generally, integral geometry, is to study a subset of Euclidean space by extracting geometric information from its various ``slices'', that is, intersections with hyperplanes. The simplest of these constructions is the \emph{support function}, which uniquely encodes a compact convex set $K\subseteq \bb R^n$ as a function $h_K:\mathbb S^{n-1} \to \mathbb R$. For each direction $v \in \bb S^{n-1}$, $h_K(v)$ records the height of the last hyperplane normal to $v$ that has a non-empty intersection with $K$. More recently, Topological Data Analysis (TDA) has seen success in studying shapes by looking at how topological invariants, such as homology groups, or the Euler characteristic, vary along a filtration of a shape. The \emph{Euler Characteristic Transform} (ECT), first defined in \cite{turner2014persistent}, has a very similar flavor to the support function, as it essentially collects topological information about all \emph{sublevel sets} --that is, parts of the shape that lie below hyperplanes-- and uniquely encodes a shape $M\subseteq \bb R^n$ as a function $\ECT(M):\mathbb S^{n-1} \times \mathbb R \to \bb Z$. 

Unsurprisingly, these two concepts can be related: the support function of a convex set can be computed from its ECT (see definition \ref{DefSupport} and lemma \ref{SuppMatchesConvex}), providing an explicit formula for its additive extension. To the author's knowledge, this fact has not been explicitly mentioned in the ECT literature, and is perhaps folklore. The main goal of this paper is to show how this interplay between TDA and convex geometry can be leveraged to extend several more definitions and results, and how these concepts show promise in applied settings.

\paragraph{Outline and main results} 
The outline of the paper is as follows. We begin in section \ref{SectPre} by giving an overview of the Euler Characteristic Transform and a couple of notions from convex geometry which will be used throughout.

Section \ref{SectExt} is the main bulk of the article, extending certain convex geometric constructions to a broader context via the ECT. In particular, generalizing the support function and the Steiner point allows us to propose two new ECT-induced pseudodistances between shapes, which we name the \emph{support distance} $\dsupp$ and the \emph{Steiner centered distance} $\dcent$, satisfying the following properties:
\begin{itemize}
	\item Both pseudodistances are invariant under common euclidean isometry. The support distance remains invariant under translation of any of the involved shapes.
	\item There is no need for introducing an additional parameter when computing pseudodistances between two shapes $M$ and $N$ with distinct Euler characteristic, contrary to classical ECT-induced distances.
	\item The Steiner centered distance extends previous definitions of ECT-induced distances when both shapes have Euler characteristic zero, or when the two shapes have the same Euler characteristic and are aligned in a particular way.
	\item The support distance is robust to noise in the form of point clouds: the distance between two shapes $M$ and $N$ remains unchanged if any amount of ``noise'' in the form of isolated points is added to one or the other.
\end{itemize}
As the second point states, these new definitions solve one recurring issue in the literature, where comparison of shapes with distinct Euler characteristic is usually dependent on some sort of cutoff parameter $R$, the radius of a ball where an integral takes place, and is highly dependent on the position of the shapes within that ball. The price to pay, however, is that they are merely pseudodistances. We give, however, certain guarantees that, at least in the plane, these pseudodistances do not mistake shapes too often, thanks to the results in section \ref{SectGeo}, which relate these distances with meaningful geometric quantities of the studied shapes, namely, their perimeter. The main result of this section is theorem \ref{PeriRecovery}, reproduced below:
\begin{theorem}[Perimeter recovery]
	Let $M \subseteq \bb R^2$ be a compact, tame shape, and let $a \in \bb R^2$. Then:
	\[ \int_{\bb S^1} \int_{\bb R} (\ECT(M)(v,t)- \chi(M) \ECT(a)(v,t)) dt dv=\Peri(M)\, .\]
\end{theorem}
Here, ``Peri'' denotes the perimeter of the shape, defined before the theorem. Thus, the perimeter of a flat shape can be computed from its ECT, without the need of reconstructing it. This result is not new, being a consequence of Crofton's formula in integral geometry (see, for instance, \cite{hug_integral_2020}), but stating it in terms of the ECT proves to be useful. In particular, geometric bounds for the support, Steiner centered, and classical ECT distances are given.

In section \ref{SectIsometry} we briefly define isometry invariant features of a shape, \emph{Steiner barcodes}, which in turn allow us to define pseudodistances between shapes that remain invariant under isometry of any of the involved shapes. We put support functions and Steiner barcodes to the test in section \ref{SectExperiment} by using them as input for training support vector machines in shape classification tasks, comparing them against training directly with ECT data. Our methods are lightweight, remain competitive in small datasets, and seem more robust than the ECT in certain contexts, although they worsen as the number of classes grows and have other shortcomings. Steiner barcodes in particular seem promising for classification tasks in unaligned datasets with few classes. In section \ref{SectionBetti} we propose another direction of study: new kinds of support functions with potentially better discriminating power, at the cost of a higher computational cost. Finally, in section \ref{SectConclusion}, we summarize our results and provide some open questions.

	\section{Preliminaries}
\label{SectPre}
\subsection{TDA and the Euler Characteristic Transform}
\label{ECT}
The Euler Characteristic Transform (from here on, \emph{ECT}), is a topological transform widely used to encode geometric datasets with connectivity information, such as meshes, simplicial complexes, or digital images (we refer the reader to \cite{munch2025invitation} for an introduction, and to \cite{turner2014persistent} for its first appearance). It relies on the Euler characteristic --an easily computable topological invariant-- and a generalization stemming from it, Euler calculus, as its theoretical backbone: we refer the interested reader to \cite{curry2012euler} for an introduction. Euler calculus is best understood in the context of o-minimal geometry, a theory that extends the usual combinatorial definition of the Euler Characteristic for a wide range of subsets of $\mathbb R^n$ --called ``definable'' or ``tame''-- while preserving additivity and multiplicativity, that is, for any two tame $M, N \subseteq \mathbb R^{n}$:
\begin{equation} \label{AddEuler}
	\chi(M)+\chi(N)=\chi(M \cup N)+\chi(M \cap N)\, ,
 \end{equation}
and 
\begin{equation}
	\chi(M \times N)=\chi(M)\cdot\chi(N) \, .
\end{equation}
Tame subsets of $\mathbb R^n$ are a wide family that includes all semialgebraic sets, and is closed under unions, intersections, and cartesian products. In particular, this family includes simplicial complexes. For tame sets that are, in addition, compact, this extended notion of the Euler characteristic agrees with the usual one. In the following, we will be dealing exclusively with compact and tame subsets of $\bb R^n$, so that we will work with the usual Euler characteristic. Informally, we will call such sets \textbf{shapes}, keeping in mind that we refer to compact sets that are, in a sense, ``well-behaved''. Tameness is not the main focus of this article, and so we will glance over the specifics; for more background, we refer the reader to \cite{coste1999introduction}.

\begin{definition}
	Let $M\subseteq \mathbb R^n$ be compact and tame. For each $t\in \mathbb R$ and $v \in \mathbb S^{n-1}$ (seeing the sphere as a subset of $\mathbb R^n$), let $M_{v,t}$ be the \textbf{sublevel} of $M$ with direction $v$ and height $t$, that is:
	\[ M_{v,t}:=\{ x \in M | \langle x, v\rangle \leq t \}\, . \]
	The \textbf{Euler Characteristic Transform (ECT)} of $M$ is the function $\ECT(M): \mathbb S^{n-1} \times \mathbb R \to \mathbb Z$ given by:
	
	\[ \ECT(M)(v,t)=\chi(M_{v,t})\, .\]
\end{definition}

\begin{remark}
	$M_{v,t}$ is the intersection of a compact and tame set $M$ and a closed and tame sublevel $\{ x \in \mathbb R^n |  \langle x , v \rangle \leq t\}$, so it is indeed compact and tame.
\end{remark}

The ECT can be visualized as follows: pick any direction $v\in \mathbb S^{n-1}$ and picture a hyperplane with normal vector $v$. Raise its height in a continuous fashion, and record the Euler characteristic of the part of $M$ that the hyperplane is leaving behind. Doing so for a fixed $v$ will result in a curve (called an \textbf{Euler curve}); the collection of all such curves is the ECT (see Figure \ref{GenericECT}).

\begin{remark}
	\label{rmkECText}
	When convenient, we extend the usual definition of the ECT by allowing it to evaluate any $v \in \bb R^{n}$: the previous definition works equally well. It is easy to show that, for $v \in \bb R^n \setminus \{0\}$, $\ECT(M)(v,t)=\ECT(\frac{v}{\|v\|}, \frac{t}{\|v\|})$, so that the new Euler curves are but rescalings of the Euler curves given by directions on the sphere, and that for $v=0$, $\ECT(M)(0,t)=\chi(M)\bb 1_{[0,\infty)}(t)$.
\end{remark}

\begin{figure}[h]
	\centering
	\includegraphics[scale=0.5]{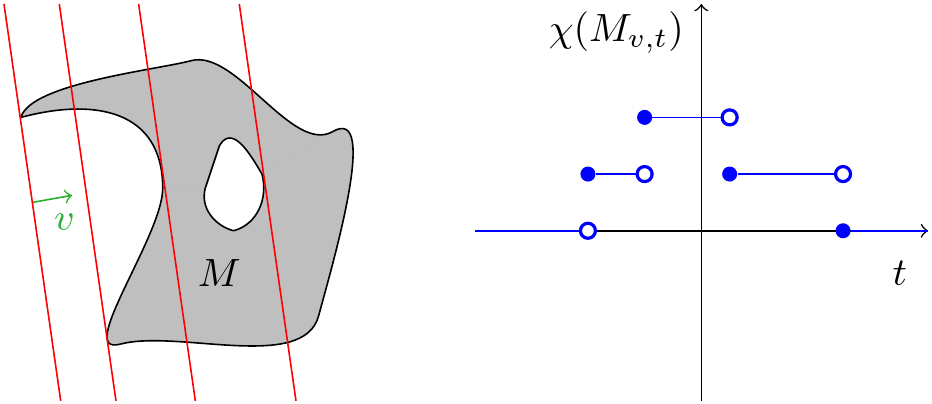}
	\caption{Euler curve of the grey shape $M$ at direction $v$. The red lines represent the heights at which there is a jump in the Euler characteristic of the sublevel sets.}
	\label{GenericECT}
\end{figure}

Perhaps surprisingly, the ECT encodes a shape without any information loss:
\begin{theorem}
	The ECT is inyective: given $M, N \subseteq \mathbb R^n$ compact and tame:
	\[ \ECT(M)=\ECT(N) \iff M=N. \]
\end{theorem}
For the proof, we refer the reader to \cite{ghrist2018persistent} or \cite{curry2022many}. Not only can we distinguish shapes by means of their ECT, but we can define a notion of distance between two shapes:
\begin{definition}
	\label{DefDistClassic}
	Let $M,N\subseteq \mathbb R^n$ be compact and tame so that $\chi(M)=\chi(N)$. We define their \textbf{(p,q)-ECT-induced distance}, with parameters $p, q \geq 1$ as:
	
	\[
		\dist{p}{q}(M,N)=\left(\int_{\bb S^{n-1}} \left( \int_\bb R |\ECT(M)(v,t)-\ECT(N)(v,t)|^p dt \right)^{\frac{q}{p}} dv\right)^\frac{1}{q}\, .
	\]
\end{definition}
\begin{remark}
	It can be proven that this is indeed a distance: the only non trivial property is that $M\neq N$ implies $\dist{p}{q}(M,N)\neq0$. This is true by virtue of the fact that, under tameness and compactness, Euler curves are step functions with a finite number of discontinuities --so two Euler curves that differ must do so in an interval--, and a known result about continuity of Euler curves as the direction parameter $v$ changes (see proposition 4.13 in \cite{curry2022many}). 
\end{remark}

The discriminating power of the ECT, in combination with good theoretical stability and approximation results (see for instance \cite{curry2022many}) have paved the way for the ECT and its many variants to find applications in several domains. These include shape classification \cite{turner2014persistent}, \cite{wang2021statistical}, analysis of black and white or grayscale medical images \cite{wang2023hypothesis}, \cite{crawford2020predicting}, or machine learning with geometric data \cite{nadimpalli2023euler}.

However, definition \ref{DefDistClassic} presents a clear disadvantage: it is only well defined when $\chi(M)=\chi(N)$, since if this is not the case, the innermost integrand does not equal $0$ at infinity. To overcome this, in many practical settings, the range $\mathbb R$ is cut short to $[-R,R]$ for a certain $R>0$ so that $M$ and $N$ are both contained in a ball with radius $R$. The chosen value of $R$ penalizes $M$ and $N$ on how much their Euler characteristics differ. However, apart from introducing an arbitrary parameter, this adds one additional artifact: the distance between two shapes $M$ and $N$ with distinct Euler characteristic may vary if both are translated by the same vector, as the next example shows.
\begin{example}
	Let $M$ be the unit disk and $N$ the unit circle. We have $\chi(M)=1$ and $\chi(N)=0$, so a cutoff radius is required: let us pick $R=2$. Notice that all proper sublevel sets of $M$ and $N$ have the same Euler characteristic, $1$, so their ECTs only differ for $t \geq 1$, where $\ECT(M)(v,t)$ takes value $1$ and $\ECT(N)(v,t)$ takes value $0$. It is then immediate that $\dist{p}{q}(M,N)=(2 \pi)^\frac{1}{q}$.
	
	Now, shift both $M$ and $N$ one unit to the left so they are still contained in the disk of radius 2. Some trigonometry now tells us that \[\ECT(M)(v,t)-\ECT(N)(v,t)=\begin{cases}
		0, t < 1-\cos v \\
		1, t \geq 1-\cos v
	\end{cases}.\] So, after the shift, we get that the new distance is $\left( \int_{\bb S^1} (1+\cos v)^\frac{q}{p} dv\right)^\frac{1}{q}$. Picking, for example, $q=2$ and $p=1$ and computing the integral, we get $(3\pi )^\frac{1}{2}$, which is certainly different from $(2\pi)^\frac{1}{2}$.
\end{example}
We will draw inspiration from convex geometry to modify definition \ref{DefDistClassic} in ways that solve this. 

\paragraph{Further properties of the ECT}
We end this brief introduction to the ECT by mentioning two particularly important facts that we will make use of:
\begin{prop}\label{add}
	The ECT is additive: given two tame, compact sets $M, N \subseteq \bb R^n$, we have, for all $v \in \bb S^{n-1}$ and $t \in \bb R$:
	\[
		\ECT(M)(v,t)+\ECT(N)(v,t)=\ECT(M \cup N)(v,t)+\ECT(M \cap N)(v,t)\, .
	\]
\end{prop}

\begin{proof}
	We have:
	\begin{align*}
		(M\cup N)_{v,t}&=\{x \in M\cup N | \langle x, v\rangle\leq t\} \\
		&=\{ x \in M | \langle x, v\rangle\leq t\} \cup \{ x \in N | \langle x, v\rangle\leq t\} \\
		&=M_{v,t} \cup N_{v,t}
	\end{align*}
	and:
	\begin{align*}
		(M\cap N)_{v,t}&=\{x \in M\cap N | \langle x, v\rangle\leq t\} \\
		&=\{ x \in M | \langle x, v\rangle\leq t\} \cap \{ x \in N | \langle x, v\rangle\leq t\} \\
		&=M_{v,t} \cap N_{v,t}.
	\end{align*}
	The result then follows from the additivity formula for the Euler characteristic, equation (\ref{AddEuler}).
\end{proof}

In fact, this is a particular case of a more general fact we will not be using: under the Euler calculus viewpoint, the ECT is a \emph{linear} integral transform.

Finally, let us see how the ECT of a shape changes after applying an euclidean isometry. Recall that any isometry $T$ of $\bb R^n$ can be written as $T(x)=L(x)+a$, where $L$ is an orthogonal map and $a \in \bb R^n$.
\begin{prop}
	\label{isoformula}
	Let $M\subseteq \bb R^n$ be compact and tame, and let $T=L+a$ be an isometry of $\mathbb R^n$. Then, for all $v \in \bb S^{n-1}$ and $t \in \bb R$:
	\[ \ECT(T(M))(v,t)=\ECT(M)(L^{-1}(v), t-\langle a,v \rangle)\, . \]
\end{prop}

\begin{proof}
	Note that since $L$ is orthogonal, $L^{-1}$ is well defined and sends vectors of $ \bb S^{n-1}$ to vectors of $\bb S^{n-1}$, so that the right-hand side is well defined. Now, given $y\in T(M)$ such that $\langle y,v \rangle\leq t$, we can write $y=T(x)=L(x)+a$ for an unique $x \in M$. Applying linearity of the dot product, we get the inequality $\langle L(x) , v\rangle +\langle a,v \rangle\leq t$, from which $\langle L(x), v\rangle \leq t-\langle a,v \rangle$. Finally, since orthogonal maps preserve dot products, we may apply $L^{-1}$ to both vectors in the left-hand side and get the same result, giving the inequality $\langle x, L^{-1}(v) \rangle \leq t-\langle a,v \rangle$ for an $x \in M$, which is the defining inequality for $M_{L^{-1}(v), t-\langle a,v \rangle}$. We have proven $T(M)_{v,t}=T(M_{L^{-1}(v), t-\langle a,v \rangle})$, and the result follows.
\end{proof}

As a consequence, we see that the issue of ECT-induced distances varying after applying a common isometry to both shapes is an artifact of the cutoff radius introduced for comparing shapes with distinct Euler characteristic:

\begin{theorem}
	ECT-induced distances are invariant under common isometry, that is, if $T$ is an isometry of $\bb R^n$ and $M,N$ are tame, compact subsets of $\bb R^n$ with the same Euler characteristic:
	\[ \dist{p}{q}(T(M), T(N))=\dist{p}{q}(M,N) \]
	for all $p, q \geq 1$.
\end{theorem}

\begin{proof}
	By definition,
	\[\dist{p}{q}(T(M), T(N))=\left(\int_{\bb S^{n-1}} \left( \int_\bb R |\ECT(T(M))(v,t)-\ECT(T(N))(v,t)|^p dt \right)^{\frac{q}{p}} dv\right)^\frac{1}{q}.\]
	Setting $T=L+a$, we rewrite the ECTs of $T(M)$ and $T(N)$ via proposition \ref{isoformula}:
	\[\dist{p}{q}(T(M), T(N))=\left(\int_{\bb S^{n-1}} \left( \int_\bb R |\ECT(M)(L^{-1}(v),t-\langle a,v \rangle)-\ECT(N)(L^{-1}(v),t-\langle a,v \rangle)|^p dt \right)^{\frac{q}{p}} dv\right)^\frac{1}{q}.\]
	The changes of variables $\tilde t= t-\langle a,v \rangle$ and $\tilde v=L^{-1}(v)$ leave the integrals unchanged, as the first is a translation and the determinant of $L^{-1}$ is $\pm 1$. We get, then:
	\[\dist{p}{q}(T(M), T(N))=\left(\int_{\bb S^{n-1}} \left( \int_\bb R |\ECT(M)(\tilde v, \tilde t)-\ECT(N)(\tilde v,\tilde t)|^p d\tilde t \right)^{\frac{q}{p}} d\tilde v\right)^\frac{1}{q},\]
	which is $\dist{p}{q}(M,N) $.
\end{proof}

\subsection{Convex geometry and the support function}
\label{convex}
Convex geometry deals with \textbf{convex sets}, that is, subsets $K$ of $\bb R^n$ so that, for any pair of points $x,y \in K$, the line segment $\{\lambda x +(1-\lambda) y | \lambda\in [0,1]\}$ is also contained within $K$. In this paper, we will work with compact convex sets. Unless otherwise stated, results and proofs appearing in this section are taken from \cite{schneider2013convex}.

A recurring idea in convex geometry is finding functions that somehow encode the geometry of a convex set. The most well-known and easy to define is the support function.

\begin{definition}
	\label{SuppDef}
	Let $K\subseteq \bb R^n$ be a compact convex set. Its \textbf{support function} is the function $h_K: \bb R^n \to \bb R$ given by:
	\[ h_K(x)=\sup_{y \in K}\langle x,y \rangle\, . \]
\end{definition}

When $v \in \bb S^{n-1}$, $h_K(v)$ records the height of $K$ when seen from direction $v$, equivalently, the height of the last hyperplane normal to $x$ that touches $K$ (see figure \ref{SuppFig}). All other values of $h_K(x)$ are but rescalings of these. It is intuitively clear, and can be proven, that this fully determines $K$ when it is compact:

\begin{figure}
	\centering
	\includegraphics[scale = 0.4]{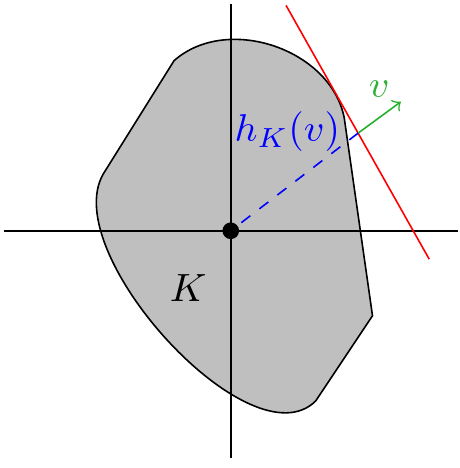}

	\caption{A convex set $K$ and the value of its support function $h_K$ at $v \in \bb S^{n-1}$. It is the height of the last hyperplane with outer normal vector $v$ that touches $K$.}
	\label{SuppFig}
\end{figure}

\begin{theorem}[Theorem 1.7.1 of \cite{schneider2013convex}]
	\label{SupportInjectiveConvex}
	The assignment $h_\bullet: K \mapsto h_K$ from compact, convex sets to functions from $\bb R^n$ to $\bb R$ is injective. Moreover, any subadditive function $f$ (that is, verifying $f(\lambda x)=\lambda f(x)$ and $f(x+y)\leq f(x)+f(y)$ for all $\lambda \geq 0$ and $x, y \in \bb R^n$) is the support function of some compact, convex set $K$. \qed
\end{theorem}

\paragraph{The support function of unions, intersections, and sums} Note that the intersection of two convex sets is always convex, but the union only sometimes is (picture two disjoint convex sets, for instance). When a shape $M$ is not convex, we may want to work with the minimal convex set that contains it, that is, its \textbf{convex closure}. It is a classical result that it consists of all convex combinations of points in $M$. We will denote it by $\operatorname{cl}(M)$. This raises a natural question: can we give formulas for the support functions $h_{K \cap L}$ and $h_{\operatorname{cl}(K\cup L)}$ in terms of $h_K$ and $h_L$?

\begin{prop}
	\label{SuppCl}
	Let $K$ and $L$ be compact convex sets. Then, for all $x\in \bb R^n$:
	\[ h_{\operatorname{cl}(K \cup L)}(x)=\max\{ h_K(x), h_L(x) \}\, . \]
\end{prop}

\begin{proof}
	All points $y\in \operatorname{cl}(K \cup L)$ can be written as a convex combination of the form $\lambda y_1+(1-\lambda) y_2$, where $y_1 \in K$, $y_2 \in L$, and $\lambda \in [0,1]$. Thus, for any $x\in \bb R^n$, we have \[\langle x,y \rangle=\langle x,\lambda y_1+(1-\lambda) y_2 \rangle=\lambda \langle x,y_1\rangle +(1-\lambda) \langle x,y_2\rangle\, .\]
	Taking the supremum over $y$ on both sides, we see that
	\[ h_{\operatorname{cl}(K \cup L)}(x)=\sup_{\lambda \in [0,1]} \lambda h_K(x)+ (1-\lambda) h_L(x)\, .\]
	As the right hand term is just a convex combination of two numbers, that supremum will be reached either when $\lambda=0$ or $\lambda=1$, depending on whether $h_L(x)$ or $h_K(x)$ is greater. Thus, the maximum of both is always picked as the value for $h_{\operatorname{cl}(K \cup L)}$.
\end{proof}

\begin{coro}[Additivity]
	\label{SuppInt}
	Suppose $K$ and $L$ are compact, convex sets such that $K \cup L$ is also convex. Then, for all $x\in \bb R^n$
	\[ h_{K\cap L}(x)=\min\{h_K(x), h_L(x)\}\, , \]
	and in particular,
	\[ h_K+h_L=h_{K \cup L}+h_{K \cap L}\, . \]
\end{coro}

\begin{proof}
	Clearly, $h_{K \cap L}\leq \min\{h_K, h_L\}$ since the supremum defining the left-hand side is taken on $K \cap L$, contained in both $K$ and $L$. Now, for a given $x \in \bb R^n$, pick $y_1 \in K$ and $y_2 \in L$ realizing the supremums: $\langle x,y_1\rangle=h_K(x)$ and $\langle x,y_2 \rangle=h_L(x)$. Since $K \cup L$ is convex, all points in the line joining $y_1$ and $y_2$ are contained in either $K$ or $L$; moreover, since $y_1$ is in $K$ and $y_2$ is in $L$, at least one point $y=\lambda y_1+(1-\lambda) y_2$ in the line must lie in $K \cap L$. The point $y\in K \cap L$ verifies 
	\[\langle x, y\rangle=\langle x, \lambda y_1+(1-\lambda) y_2 \rangle=\lambda \langle x, y_1 \rangle +(1-\lambda) \langle x,y_2\rangle=\lambda h_K(x)+(1-\lambda) h_L(x) \, ,\]
	and so we obtain $\langle x, y \rangle\geq \min\{ h_K(x), h_L(x) \}$, thus $h_{K \cap L}(x) \geq \min\{h_K(x), h_L(x)\}$.
	
	The second formula follows, since $h_{K \cup L}=h_{\operatorname{cl}(K \cup L)}=\max\{h_K, h_L\}$ by proposition \ref{SuppCl}.
\end{proof}

We can also compute the support function of the sum of two convex sets, as defined below:
\begin{definition}
	Let $K$ and $L$ two convex sets. Their \textbf{Minkowski sum} $K+L$ is the convex set defined by all possible sums of points of $K$ with points of $L$, that is: 
	\[ K+L:=\{ k+l | k \in K, l \in L \}\, . \]
\end{definition}

\begin{prop}
	\label{MinkAdd}
	Let $K$ and $L$ two compact convex sets. Then,
	\[ h_{K+L}=h_K+h_L\, . \]
\end{prop}

\begin{proof}
	This follows at once from $\langle x, k+l \rangle=\langle x, k \rangle+\langle x, l \rangle$.
\end{proof}

The support function is useful in a great deal of scenarios: we'll outline below how it can be used to define distances between convex sets and its many integral formulas. We will give no proofs in the following.

\paragraph{$L^p$ norms on convex sets}
\label{LpNorms}

In a similar vein as the ECT distances defined earlier, one can define $L^p$ norms and distances between convex sets by means of their support functions:

\begin{definition}
	\label{LpNorm}
	Let $K$ be a compact, convex set, and let $p\in [1, \infty]$. Its \textbf{$L^p$ norm} is defined as
	\[ \|K\|_p:=\left(\int_{\bb S^{n-1}} |h_K(v)|^p dv\right)^\frac{1}{p}. \]
	Given two compact, convex sets $K,L$, their \textbf{$L^p$ distance} is similarly defined as 
	\[ \delta_p(K,L):=\left(\int_{\bb S^{n-1}} |h_K(v)-h_L(v)|^p dv\right)^\frac{1}{p}. \]
\end{definition}

\begin{remark}
	When $p=\infty$, we understand the formulas above as
	\[ \|K\|_\infty = \sup_{v \in \bb S^{n-1}} |h_K(v)| \]
	and 
	\[ \delta_\infty(K,L) = \sup_{v \in \bb S^{n-1}} |h_K(v)-k_L(v)|\, , \]
	as is custom, and we will do so in the remainder of the article.
\end{remark}
These distances were more thoroughly examined by Vitale in \cite{vitale1985lp}. We mention an important special case, relating the support function to the Hausdorff metric:
\begin{theorem}[Lemma 1.8.14 of \cite{schneider2013convex}]
	\label{HausdorffDist}
	The $\delta_\infty$ distance between two compact, convex sets matches with their Hausdorff distance, that is:
	\[ d_H(K,L)=\max_{v \in \bb S^{n-1}} |h_K(v)-h_L(v)|.\] 
\end{theorem}

\paragraph{Integral formulas for the support function}

A great deal of geometric invariants of convex sets can be obtained via integrals where the support function plays a role. For our purposes, we single out three of them, referring the reader to \cite{schneider2013convex} for proofs and further formulas. The first is quite trivial: notice that, given a direction $v \in \bb S^{n-1}$, the value $h_K(v)+h_K(-v)$ is precisely the width of $K$ when seen from direction $v$, as we are just adding the (signed) distances to the origin of the lower-most and higher-most points of $K$ when seen from direction $v$. Thus:
\begin{theorem}
	\label{meanwidth}
	The \textbf{mean width} of a compact convex set $K \subseteq \bb R^n$ is given by
	\[ \overline W(K)=\frac{2}{\mu(\bb S^{n-1})}\int_{\bb S^{n-1}} h_K(v) dv \, ,\]
	where $\mu(\bb S^{n-1})$ is the surface area of the $(n-1)$-sphere. \qed
\end{theorem}

Secondly, one can relate the support function to a certain distinguished point in a convex set, which will take a central role in this paper. We first give a loose geometric definition of it follows:
\begin{definition}
	Let $K$ be a compact convex set. Its \textbf{Steiner curvature centroid} or simply \textbf{Steiner point} is defined as the center of mass of its boundary when weighed by its Gaussian curvature, that is:
	\[ s(K):= \frac{\int_{\partial K} xC_0(x) dx}{\int_{\partial K} C_0(x) dx}\, , \]
	where $C_0(x)$ denotes the Gaussian curvature of $\partial K$ at the point $x$.
\end{definition}
\begin{remark}
	
	When $K$ is a convex polygon in the plane, $s(K)$ is given by the weighted sum of all vertices of the polygon, weighted by their external angles. The Steiner point is also sometimes called the \emph{Steiner curvature centroid} to distinguish it from another distinguished point particular to triangles which shares the same name.
\end{remark}

The Steiner point can be computed in another, simpler, way by means of the support function. This is the definition we will keep in mind for our purposes.

\begin{theorem}
	\label{SteinerIntegralFormula}
		Let $K\subseteq \bb R^n$ be a compact convex set. Then,
		\[ s(K)=\frac{1}{\mu(\bb B^{n})}\int_{\bb S^{n-1}} v \cdot h_K(v) dv\, , \]
		where $\mu(\bb B^{n})$ is the volume of the unit $n$-ball.
\end{theorem}

Finally, there are a couple interesting integral formulas involving the support function that hold in the plane. For starters, it is a remarkable fact of convex geometry that, in the plane
\begin{equation}
	\label{PeriSupport}
\int_{\bb S^{1}} h_K(v) dv =\Peri(K)\, ,
\end{equation}
where ``$\Peri$'' denotes the perimeter. Equating this with the general result from theorem \ref{meanwidth}, we get the following:

\begin{coro}[Barbier's theorem]
	\label{Barbier}
	Let $K\subseteq \bb R^2$ be a compact convex set in the plane. Then,
	\[ \Peri(K)=\pi \overline W(K)\, . \]
	In particular, all flat shapes of constant width $w$ have the same perimeter, $\pi w$.
\end{coro}

Finally, under some differentiability assumptions on the boundary on $K$, we can relate the support function of a convex set in the plane with its area and curvature.

\begin{theorem}[Equations (4) and (5) from \cite{malagoli2009optimal}]
	\label{SupportArea}
	Let $K\subseteq \bb R^2$ be a compact convex set such that $h_K(v)$ is twice differentiable over $\bb S^{n-1}$, for each $v \in \bb S^{n-1}$ there is a unique point $x(v) \in \partial K$ such that $h_K(v)=x(v)$, and $\partial K$ has a well defined radius of curvature $\rho(v)$ at each point $x(v)$. Then,
	\[ h_K(v)+h_K''(v)=\rho(v)\, , \]
	and
	\[ \operatorname{Area}(K)=\frac{1}{2}\int_{\bb S^1} h_K(v) \rho(v) dv\, .\]
\end{theorem}

\begin{remark}
	All these results are special cases of many more, involving quantities in convex geometry that generalize the volume and surface area, dubbed \emph{mixed volumes} (some special cases of these involve so-called \emph{quermassintegrals}). We have only mentioned the few that appear to have a straightforward connection with the ECT and familiar geometric features.
\end{remark}

\subsection{Valuations and extending them beyond convex sets}
\label{valuations}

The support function, as written in definition \ref{SuppDef}, does not require convexity to be defined. This begs the question of whether its obvious extension (just replicating its definition) over a greater class of shapes, is a ``good'' extension, in some sense. To illustrate how an extension like this can sometimes be uninteresting, let us first take a look at a simpler function over convex sets: the Euler characteristic. For a compact convex set $K$, its Euler characteristic is defined as:
\[ \chi(K)=\begin{cases}
	1, \text{ if } K \neq \emptyset \\
	0, \text{ if } K = \emptyset
\end{cases}. \]
This agrees with the topological Euler characteristic, as any nonempty convex set is contractible. If we just extend this definition to a greater class of shapes, we end up with a somewhat trivial function, that only checks if a set is empty or not.

A more interesting approach is the following: we first observe that, for two nonempty convex sets, \emph{if $K \cup L$ is convex}, then $K \cap L$ must be nonempty (otherwise, $K \cup L$ would consist of two disjoint pieces), and thus the equation \[\chi(K)+\chi(L)=\chi(K \cup L)+\chi(K \cap L)\] holds, as both sides equal two. This additivity formula will also hold if either or both of $K$ and $L$ are empty. This formula is summarized by saying that the Euler characteristic is a \textbf{valuation} over convex sets. This is a very desirable property, and giving an extension of the Euler characteristic to a greater family of sets while preserving this formula is simple: given $K$ and $L$ convex, \emph{whether or not their union is convex}, we define:
\[ \chi(K \cup L):=\chi(K) +\chi(L) -\chi(K \cap L)\, .\]
So, for instance, the union of two disjoint convex sets has Euler characteristic $2$. The extension to unions of three sets is done similarly, however, one must check two facts: that the definition is independent of the decomposition into convex sets given, and that a general additivity formula holds. A theorem by Groemer \cite{groemer1978extension} asserts the following:
\begin{theorem}[Groemer]
	\label{Groemer}
	Let $\lambda$ be a valuation over a class $\mathcal S$ of sets that is closed under intersections. Then, $\lambda$ admits a unique additive extension to the class of finite unions of members of $\mathcal S$ if and only if $\lambda$ is generally additive, that is, if the equation
	\[ \lambda(X_1 \cup ... \cup X_m)=\sum \lambda (X_i) -\sum_{j<i} \lambda (X_j \cap X_i)+\sum_{i<j<k} \lambda(X_i \cap X_j \cap X_k) -\dots \]
	holds whenever $X_1,...,X_m$ and $X_1 \cup ... \cup X_m$ are in $\mathcal S$.
\end{theorem}

It can be seen that the Euler characteristic is indeed generally additive, as is the support function. Both valuations then admit a unique extension to the family of finite unions of compact convex sets. The general additivity formula allows one to evaluate a generally additive valuation when a decomposition of the space into convex sets is known. In addition, we can generalize many constructions of the last section, like, say, the Steiner point, to this greater class of shapes, and they too will verify additivity formulas. In the next section, we will provide some explicit formulas for the extension of the support function.
	
	\section{Extensions of convex-geometric constructions via the ECT}
\label{SectExt}
\subsection{Extending the support function}

It is hopefully clear how both the ECT and the support function encode similar kinds of information. Here we make the connection between these two concepts explicit:

\begin{definition}
	\label{DefSupport}
	Let $M\subseteq \bb R^n$ be compact and tame. Its \textbf{support function} is defined as the function $h_M: \bb R^n \to \bb R$ given by:
	\[ h_M(x)=\int_{\bb R} \left(\ECT(M)\left(-x,t\right)-\chi(M)\bb 1_{[0, +\infty)}(t)\right) dt\, .\]
\end{definition}

\begin{lemma}
	\label{SuppMatchesConvex}
	The support function of a tame, compact convex set $K$ agrees with its usual definition.
\end{lemma}
\begin{proof}
	All sublevel sets of a compact convex set $K$ are either empty, or are themselves compact and convex, being the intersection of two closed convex sets, one of them compact. Thus:
	\[\ECT(K)(-x,t)=\begin{cases}
		0, \text{ if } \forall y \in K, t< \langle -x, y \rangle \\
		1, \text{ if } \exists y \in K \text{ s.t. }t \geq \langle -x, y \rangle
	\end{cases}=\begin{cases}
		0, \text{ if } t< -h_K(x)  \\
		1, \text{ if } t\geq -h_K(x)
	\end{cases}.\]
	If $h_K(x)\geq 0$, the integrand is only non zero in the interval $[-h_K(x), 0]$, taking value $1$, and if $h_K(x)<0$, it will be nonzero in the interval $[0,-h_K(x)]$, taking value $-1$. In either case the integral evaluates to $h_K(x)$.
\end{proof}

\begin{prop}
	\label{SuppAddExt}
	The support function defined above is an additive extension of the support function to the class of compact, tame $n$-euclidean shapes. Moreover, it restricts to the unique extension of the support function to finite unions of tame, compact, convex sets.
\end{prop}

\begin{proof}
	We have just seen that the definition matches with the support function for convex sets; it only remains to see that it is additive. This follows immediately from additivity of the ECT (proposition \ref{add}) and of the Euler characteristic (equation (\ref{AddEuler})) along with linearity of the integral. The extension when restricted to unions of convex sets is unique by theorem \ref{Groemer}.
\end{proof}

\begin{remark}
	There are two useful ways one can think of the term $\chi(M)\bb 1_{[0,+\infty)}(t)$ appearing in the integral. The first is the value $\ECT(M)(0,t)$, as we have stated in remark \ref{rmkECText}. The second one starts by noticing that $\bb 1_{[0, +\infty)}(t)=\ECT(a)(v,t)$ for all $v\in \bb R^n$ and $t\in \bb R$, where $a$ is a point set at the origin. Thus, we can interpret the integrand as the difference of ECTs of $M$ and of a ``weighted point'' set at the origin, with weight equal to the Euler characteristic of $M$. Equivalently, we could picture a very small copy of $M$ set at the origin.
\end{remark}

\begin{remark}
Another explicit formula for the additive extension of the support function to the class of unions of convex sets was given by Mani in \cite{mani_angle_1971} and studied further by Hadwiger and Schneider in \cite{hadwiger1971vektorielle}:
\[ h_M(x)=\sum_{t \in \mathbb R}  t (\chi(M \cap H_{x, t})-\lim_{\mu \to t^+} \chi(M \cap H_{x, t}))\, ,\]
where $H_{x, t}=\{ y \in \bb R^n | \la x,y\ra = t \}$. This can be seen to be equal to our definition. (To see this, write $\chi(M \cap H_{v,t})=\chi(M_{v,t})+\chi(M_{-v,-t})-\chi(M)$ by additivity and apply right-continuity of Euler curves. The result is easily related to the integral in definition \ref{DefSupport}.) 
\end{remark}

We could, at this point, take inspiration from the $L^p$ distances defined in section \ref{LpNorms} and naively define support distances between two shapes by the same procedure. However, the same issue we had for the classical ECT-induced distance between shapes with distinct Euler characteristic arises: distances change whenever the same linear isometry is applied to both compared shapes. This time, the issue is different: by subtracting $\chi(M) \bb 1_{[0, \infty)}(t)$ from each Euler curve, we have singled out the origin as a special point. We can do better than this by defining a central point for a shape, and taking that as a reference instead of the origin. Since we will be moving shapes around to realize this strategy, let us see how the support function changes under isometries:

\begin{prop}
	\label{IsoSupport}
	Let $M\subseteq \bb R^n$ be a compact, tame, shape, and let $T=L+a$ be an an euclidean isometry. Then, for all $x \in \bb R^n$:
	\[ h_{T(M)}(x)=h_M(L^{-1}(x))+\chi(M) \la x, a \ra \, . \]
\end{prop}

\begin{proof}
	Using proposition \ref{isoformula} and introducing the change of variables $\tilde t = t+\la x,a \ra$, we write
	\begin{align*} h_{T(M)}(x) &= \int_{\bb R} (\ECT(M)(-L^{-1}(x), t+\la x,a \ra)-\chi(M)\bb 1_{[0, +\infty)}(t)) dt \\
	& = \int_{\bb R} (\ECT(M)(-L^{-1}(x), \tilde t)-\chi(M)\bb 1_{[\la x,a \ra, +\infty)}(\tilde t)) d\tilde t \\
	& = \int_{\bb R} (\ECT(M)(-L^{-1}(x), \tilde t)-\chi(M)(\bb 1_{[0, +\infty)}(\tilde t)+f(\tilde t)) d\tilde t\\
	& = \int_{\bb R} (\ECT(M)(-L^{-1}(x), \tilde t)-\chi(M)\bb 1_{[0, +\infty)}(\tilde t)) d \tilde t-\chi(M)\int_{\bb R}f(\tilde t) d\tilde t \, , \end{align*}
	where $f(\tilde t)= \begin{cases}
		\bb 1_{[\la x,a \ra, 0)}(\tilde t), \text{ if } \la x,a \ra<0 \\
		-\bb 1_{[0,\la x,a \ra)}(\tilde t), \text{ if } \la x, a\ra >0
	\end{cases}$. The integral of $f$ over $\bb R$ is obviously $-\la x, a\ra$, yielding the result.
\end{proof}

\subsection{The (weighted) Steiner point}
Having defined the support function, it is a simple matter to define the Steiner point of a shape $M$ by taking to heart the formula given in theorem \ref{SteinerIntegralFormula}. However, we will introduce a change: dividing the result by the Euler characteristic of the shape.

\begin{definition}
	Let $M\subseteq \bb R^n$ be compact and tame with $\chi(M) \neq 0$. Its \textbf{weighted Steiner point} or simply \textbf{Steiner point} is defined by:
	\[ s(M):=\frac{1}{\chi(M) \mu(\bb B^{n})} \int_{\bb S^{n-1}} v \cdot h_M(v) dv \, . \]
	If $\chi(M)=0$, we set its Steiner point as the origin. 
\end{definition}

\begin{remark}
	\label{RemarkUnweighted}
	Although our main focus will be the \emph{weighted} Steiner point, it also will be useful to us to consider the \textbf{\emph{unweighted} Steiner point}, defined for any compact, tame shape $M$ as:
	\[ \frac{1}{\mu(\bb B^n)} \int_{\bb S^{n-1}} v \cdot h_M(v) dv\, .\]
	Since for $\chi(M) \neq 0$ this is just $\chi(M)s(M)$, we will denote this unweighted point by $\chi(M)s(M)$, keeping in mind that if $\chi(M)=0$, we actually mean the result of the integral above.
\end{remark}

Setting the Steiner point as the origin for a shape of Euler characteristic zero is an arbitrary choice which will amount to nothing for practical purposes, as will become clear shortly. 
The reason for dividing by the Euler characteristic is so we have the following:

\begin{prop}
	\label{SteinerIsometry}
	Let $M$ be compact and tame with $\chi(M)\neq 0$. If $T$ is an euclidean isometry, then
	\[ s(T(M))=T(s(M)). \]
\end{prop}

\begin{proof}
	By definition, and using proposition \ref{IsoSupport},
	\begin{align*} s(T(M))&=\frac{1}{\chi(T(M)) \mu(\bb B^n)} \int_{\bb S^{n-1}} v\cdot h_{T(M)}(v) dv \\
	&=\frac{1}{\chi(M) \mu(\bb B^n)} \int_{\bb S^{n-1}} v\cdot (h_M(L^{-1}(v))+\chi(M)\la v,a\ra) dv \\
	&=\frac{1}{\chi(M) \mu (\bb B^n)} \left(\int_{\bb S^{n-1}} v \cdot h_M(L^{-1}(v)) dv +\chi(M)\int_{\bb S^{n-1}} v \cdot \la v,a \ra dv\right).
	 \end{align*} 
	 By introducing the linear change of variables $\tilde v:=L^{-1}v$ on the first integral, we can rewrite it as $\int_{\bb S^{n-1}} L(\tilde v \cdot h_M(\tilde v)) d\tilde v$, and since $L$ is linear, it can be taken out of the integral, yielding $\chi(M)\mu(\bb B^n)L(s(M))$. The other integral is a standard computation and equals $\mu(\bb B^n)a$. All in all:
	 \[ s(T(M))=\frac{1}{\chi(M) \mu(\bb B^n)}(\chi(M) \mu(\bb B^n)L(s(M))+\chi(M)\mu(\bb B^n)a)=L(s(M))+a=T(s(M)). \]
\end{proof}

\begin{example}[The Steiner point of a point cloud]
	With a similar computation as above, one can check that the Steiner point of a point cloud --that is, a finite number of isolated points-- is its usual center of mass. Thus, the weighted Steiner point is a generalization of both the Steiner point of a convex set, and of the center of mass of a point cloud.
\end{example}

The Steiner point also verifies a weighted version of the additivity formula:

\begin{prop}
	\label{SteinerWeightAdd}
	For $M,N\subseteq \bb R^n$ compact and tame:
	\[ \chi(M)s(M)+\chi(N)s(N)=\chi(M \cup N)s(M \cup N)+\chi(M \cap N)s(M \cap N)\, ,\]
	where we are using the convention of remark \ref{RemarkUnweighted} if any shape has Euler Characteristic zero.
\end{prop}

\begin{proof}
	Since the support function is additive, we have \[\int_{\bb S^{n-1}} v (h_M(v)+h_N(v)) dv=\int_{\bb S^{n-1}} v (h_{M \cup N}(v)+h_{M \cap N}(v)) dv,\] from which the result follows at once.
\end{proof}

\paragraph{The Steiner point as a minimizer}
Here is another perspective on the Steiner point:

\begin{prop}
	\label{SteinerMinimizer}
	The Steiner point $s(M)$ of a compact, tame shape $M\subseteq \bb R^n$ with $\chi(M) \neq 0$ is, among all points $a\in \bb R^n$, the unique minimizer of the following value:
	\[ \int_{\bb S^{n-1}} \left( \int_{\bb R} (\ECT(M)(-v,t)-\chi(M)\ECT(a)(-v,t)) dt\right)^2 dv\, .\]
\end{prop}

\begin{proof}
	By adding and subtracting $\chi(M)\bb 1_{[0,+\infty)}(t)$ in the innermost integral, the expression can be rewritten as $$\int_{\bb S^{n-1}}(h_M(v)-\chi(M) \langle v,a \rangle)^2 dv\, .$$ The minimizing problem is thus reduced to finding the best linear approximation of the function $h_M(v)$ in the least squares sense. Expanding the square and a standard minimization argument recovers the formula given for the Steiner point.
\end{proof}
In this light, the Steiner point is a central point in the following sense: it is the point ``closest'' to $M$ in a certain distance, given by an $L^2$ norm\footnote{Before the connection with the classical Steiner point was made apparent, the author suggested the name ``cECTroid'' for this distinguished point. It is probably for the best that a better name already exists.}. The natural question is if by varying the norm we can define other interesting distinguished points of a shape. In general, however, these points may not be unique. For instance, for a simply connected shape in $\bb R^2$, any point lying on its convex closure minimizes the analogous $L^1$ norm (this will follow from theorems \ref{PeriRecovery}, \ref{PeriLowBound}, and  lemma \ref{SimpCompPositive}). 

\subsection{The Steiner support and new pseudodistances}

With the Steiner point in our toolkit, we can now provide a new kind of support function, better suited for our purposes, since it doesn't single out the origin as a special point, and thus yields norms and distances between shapes that are invariant under common euclidean isometries.

\begin{definition}
	Let $M\subseteq \bb R^n$ be compact and tame. We define its \textbf{Steiner support function} as the function $S_M: \bb  R^n\to \bb R$ given by:
	\[ S_M(x) = h_M(x)-h_{\chi(M)s(M)}(x) =h_M(x)- \la x, \chi(M)s(M) \ra\, .\]
	Whenever $\chi(M) \neq 0$, we may use the following equivalent definition:
	\begin{align*} S_M(x):&=\int_{\bb R} (\ECT(M)(-x,t)-\chi(M) \ECT(s(M))(-x,t)) dt \\
		&=\int_{\bb R} (\ECT(M)(-x,t)-\chi(M) \bb 1_{[-\langle x,s(M) \rangle, +\infty)}(t)) dt\, .
	\end{align*}
\end{definition}

	It is straightforward to see that both definitions agree in the case $\chi(M) \neq 0$, which also agree with the support function of the centered shape $M-s(M)$. 

\begin{prop}
	\label{LinEquiv}
	The Steiner support function is linear-equivariant with respect to euclidean isometries, that is, for all compact tame shapes $M\subseteq \bb R^n$ and all isometries $T=L+a$, we have
	\[ S_{T(M)}(x)=S_{M}(L^{-1}(x))\, .\]
	In particular, for translations,
	\[ S_{M+a}=S_M\, .\]
\end{prop}

\begin{proof}
	For $\chi(M)\neq 0$, we have $s(T(M))=T(s(M))$ and $h_{T(M)}(v)=h_M(L^{-1}(x))+ \la x, \chi(M)a\ra$ by propositions \ref{SteinerIsometry} and \ref{IsoSupport}. Thus,
	\begin{align*} 
		S_{T(M)}(x)&=h_M(L^{-1}(x))+\la x,\chi(M)a\ra-\la x, \chi(M)(L(s(M))+a) \ra \\
		&= h_M(L^{-1}(x))-\la x, L(\chi(M)s(M)) \ra \\
		&= h_M(L^{-1}(x))-\la L^{-1}(x), \chi(M)s(M) \ra \\
		&=S_M(L^{-1}(x))\, .
	\end{align*}
	
	If $\chi(M)=0$, notice that by proposition \ref{IsoSupport} we have $h_{T(M)}(x)=h_{M}(L^{-1}(x))$. Therefore:
	\begin{align*} 
		\chi(T(M))s(T(M)) &= \frac{1}{\mu(\bb B^n)}\int_{\bb S^{n-1}} v \cdot h_{T(M)}(v) dv\\
		&=\frac{1}{\mu(\bb B^n)}\int_{\bb S^{n-1}} v \cdot h_M(L^{-1}(v)) dv\\ &= L\left(\frac{1}{\mu(\bb B^n)}\int_{\bb S^{n-1}} \tilde v \cdot h_M(\tilde v) d\tilde v\right) \\&= L(\chi(M)s(M))\, , \end{align*}
	where $\tilde v = L^{-1}(v)$. We conclude that $S_{T(M)}(x)=h_M(L^{-1}(x))-\la x, L(\chi(M)s(M)) \ra$, and from there
	we proceed as in the case $\chi(M)\neq0$.
\end{proof}
\begin{prop}
	\label{SteinerSuppAdd}
	The Steiner support function is additive:
	\[ S_M(x)+S_N(x)=S_{M \cup N}(x)+S_{M \cap N}(x). \]
\end{prop}

\begin{proof}
	This follows immediately from proposition \ref{SuppAddExt}, additivity of the support function, and proposition \ref{SteinerWeightAdd}, additivity of the unweighted Steiner point.
\end{proof}

\begin{remark}
	The usual support function cannot distinguish between $K \sqcup L$ and $K+L$ for $K$ and $L$ disjoint convex shapes, as both equal $h_K +h_L$ (see proposition \ref{MinkAdd}). However, the Steiner support function can in general: since $\chi(K + L)=1$ while $\chi(K \sqcup L)=2$, it turns out that $s(K+L)=s(K)+s(L)$ while $s(K \sqcup L)=\frac{s(K)+s(L)}{2}$. This means that the Steiner support function of $K+L$ is the support function of $K+L-(s(K)+s(L))$, while the one of $K \sqcup L$ is the support function of $K \sqcup L-\left( \frac{s(K)+s(L)}{2}\right)$, which equals the one of $K + L-\left( \frac{s(K)+s(L)}{2}\right)$. By theorem \ref{SupportInjectiveConvex}, the functions are distinct whenever these two convex sets are distinct, that is, whenever $s(K)+s(L) \neq 0$. On the other hand, the Steiner support function cannot distinguish certain shapes that the support function can, as we will see later with point clouds.
\end{remark}

\begin{remark}
	In the case of shapes with Euler characteristic zero, one could bypass subtracting $\la x, \chi(M)s(M) \ra$, defining $S_M(x)$ as $h_M(x)$, and still get a function with the same behavior under euclidean isometries. This has the advantage of letting the alternate definition
	\[ S_M(x) = \int_{\bb R}(\ECT(M)(-x,t)-\chi(M) \bb 1_{[-\la x, s(M) \ra,+\infty)}(t)) dt\] become valid for all shapes. However, doing so breaks additivity, hence why we added this ``correction term''. Nonetheless, this alternate definition may be of use.
\end{remark}

\paragraph{New ECT-based pseudodistances}
We arrive at our main definition: new ``seminorms'' and pseudodistances for compact, tame shapes that are invariant under common isometry. For ease of writing, we introduce the following notation for the \textbf{Steiner centered ECT}:
\[ \overline{\ECT}(M)(v,t):= \ECT(M)(v,t)-\chi(M) \ECT(s(M))(v,t) \, . \]
Note that if $\chi(M) \neq 0$, we have \[S_M(v)=\int_{\bb R} \overline{\ECT}(v,t) dt\, .\]

\begin{definition}
	\label{DefNormSupp}
	Let $M$ be compact and tame, and let $p,q\in [1, +\infty]$. We define its \textbf{$p$-support norm} as
	\[ \normsupp{M}_p:=\left(\int_{\bb S^{n-1}} |S_M(v)|^p dv\right)^{\frac{1}{p}}. \]
	Likewise, we define its \textbf{$(p,q)$-Steiner centered norm} as
	\[\normcent{M}_{p,q}:=\left(\int_{\bb S^{n-1}} \left( \int_{\bb R}|\overline{\ECT}(M)(v,t)|^p dt \right)^\frac{q}{p} dv\right)^\frac{1}{q} .\]
\end{definition}

\begin{definition}
	\label{DefDistSupp}
	Let $M,N$ be compact and tame shapes, and let $p,q\in [1, +\infty]$. We define their \textbf{$p$-support distance} $\dsupp_p(M,N)$ as
	\[ \dsupp_p(M,N):=\left(\int_{\bb S^{n-1}} |S_M(v)-S_N(v)|^p dv\right)^\frac{1}{p}.\]
	Likewise, we define their \textbf{$(p,q)$-Steiner centered distance} $\dcent_{p,q}(M,N)$ as
	\[ \dcent_{p,q}(M,N):=\left( \int_{\bb S^{n-1}} \left( \int_{\bb R} |\overline{\ECT}(M)(v,t)-\overline{\ECT}(N)(v,t)|^p dt\right)^{\frac{q}{p}} dv\right)^{\frac{1}{q}}. \]
\end{definition}

\begin{coro}
	The $p$-support and $(p,q)$-Steiner centered norms are invariant under isometries. The $p$-support and $(p,q)$-Steiner centered distances are invariant under common isometry. The $p$-support distance between two shapes $M$ and $N$ is unchanged when either shape is translated.
\end{coro}

\begin{proof}
	This follows at once from proposition \ref{LinEquiv} for the $p$-support norm and distance. The $(p,q)$-Steiner centered case is analogous and follows from the behavior of the Steiner point and the ECT under isometries (propositions  \ref{SteinerIsometry} and \ref{isoformula}).
\end{proof}

Note that if $M$ and $N$ share the same Steiner point and Euler characteristic, $\dcent_{p,q}(M,N)$ coincides with the classical ECT-induced distance, see definition \ref{DefDistClassic}. The same is true if both shapes have Euler characteristic zero, regardless of their position. 

Next, we point out some trivial relations between these norms and distances:
\begin{prop}
	\label{DistanceBounds}
	Let $p,q \in [1, +\infty]$. For any compact, tame shape $M$ with $\chi(M) \neq 0$, let $R$ be the radius of any ball that fully encloses $M$ and $s(M)$. Then:
	\[ \normsupp M_q\leq (2R)^{1-\frac{1}{p}}\normcent M_ {p,q},\]
	Likewise, for any compact, tame shapes $M$ and $N$ with $\chi(M), \chi(N) \neq 0$, let $R$ be the radius of any ball that fully encloses $M$, $N$, $s(M)$ and $s(N)$. Then:
	\[ \dsupp_q(M,N)\leq (2R)^{1-\frac{1}{p}}\dcent_{p,q}(M,N).\]
\end{prop}
\begin{proof}
	Translate the shape(s) so that the center of the enclosing ball is the origin. The integration range $\bb R$ can then be safely cut short to $[-R, R]$, as all hyperplanes below height $-R$ have not yet reached any of the involved shapes, and all above height $R$ have surpassed them all. The statement then follows from the Hölder inequalities \[ \left|\int_{-R}^R f\right| \leq \int_{-R}^R |f| \leq\left(\int_{-R}^R |f|^p\right)^\frac{1}{p}\left( \int_{-R}^R 1^{1-\frac{1}{p}}\right)^{1-\frac{1}{p}} =  \left(\int_{-R}^R |f|^p\right)^\frac{1}{p} (2R)^{1-\frac{1}{p}}\]
	applied to $f=\overline{\ECT}(v,\bullet)$.
\end{proof}

\begin{remark}
	\label{RmkAlternativeCentered}
	Similar bounds hold for shapes with Euler characteristic zero, but with an added term to account for $\la v, \chi(M)s(M)\ra$. This hints that perhaps the definition of the Steiner centered ECT for the case $\chi(M)=0$ should be changed to 
	\[ \overline{\ECT}(M)=\ECT(M)+\ECT(a)-\ECT(\chi(M)s(M))\, , \]
	where $a$ is a point placed at the origin. This makes it so that $S_M(v)= \int_\bb R \overline{\ECT}(-v, t) dt$ in this case as well, and makes the bounds work in all cases.
\end{remark}

As usual, variations of these definitions can be made if one is interested in distances up to isometry, by taking the infimum of all distances when one of the shapes is transformed under an isometry. We will later describe an alternative approach in definition \ref{DefBarcodes}.

It is important to keep in mind that despite calling them distances for short, these are merely pseudodistances: indeed, the Steiner point of a single point is itself, which means its Steiner support function is identically zero. This means any two points have zero support and Steiner centered distance. More is true: any point cloud (disjoint union of points) has identically zero Steiner support, as follows from proposition \ref{SteinerSuppAdd} and the fact that any point has identically zero Steiner support function. This can be seen as a strange quirk, or as a strength of the support distance, as it implies it is completely unaffected by noise in the form of point clouds:

\begin{theorem}[Support distance robustness under point cloud noise]
	\label{ThmPointCloud}
	Let $M$ be a compact tame shape, and let $P$ be a point cloud disjoint from $M$. Then:
	\[ S_{M \sqcup P}=S_M\, , \]
	and for any other compact, tame shape $N$:
	\[\dsupp_p(M \sqcup P, N)=\dsupp_p(M,N)\, .\]
\end{theorem}

\begin{proof}
	As proven by the discussion above, $S_P$ is identically zero. The statement follows immediately from additivity, proposition \ref{SteinerSuppAdd}.
\end{proof}
Finally, the support distance also verifies an interesting additive property:
\begin{coro}
	\label{IterativeDist}
	Let $M$ and $N$ be compact, tame shapes. Then:
	\[ \dsupp_p(M\cup N, M)= \dsupp_p(N, M\cap N)\, .\]
\end{coro}

\begin{proof}
	This follows at once from additivity, proposition \ref{SteinerSuppAdd}.
\end{proof}

This last result leads to a computational advantage when working with filtered shapes: suppose we have a growing sequence of shapes $M_1 \subseteq ... \subseteq M_n$ and we wish to compute the distance between two stages $M_i \subseteq M_j$ of the sequence. Writing $M_j=M_i \cup N$ for some compact, tame set $N$ (representing the ``growth'' achieved from $M_i$ to $M_j$), corollary \ref{IterativeDist} ensures that we may replace this calculation by computing the support distance between $N$ and $N \cap M_i$. Since $N$ is the growth from one stage to the other, $N$ may be a much simpler shape than $M_i$ and $M_j$. This is particularly true for sequences arising from filtered simplicial complexes, for which $N$ and $N \cap M_i$ may be just a few simplices.

\subsection{Non-additive extensions}

Taking a step back, our general strategy has relied on transforming Euler curves into something we could integrate along each direction. We have done this by subtracting a step function whose value at infinity was the Euler characteristic of our shape. This approach preserves additivity, which is desirable from a theoretical and computational standpoint, but causes different shapes to share the same support function, mistaking them as the same. There are more ways to take Euler curves into more amenable objects: these types of strategies are known in the TDA literature as \emph{hybrid transforms}. Essentially, Euler curves are multiplied by a kernel: a real valued function which make them integrable across the whole of $\bb R$. In \cite{lebovici_hybrid_2024}, example 7.13, the authors define a hybrid transform reminiscent of the Laplace transform and relate it to an invariant called the persistent magnitude, defined in \cite{govc2021persistent}, where the connection with the Laplace transform was made clear. As it turns out, these strategies also have their place in the support function context, allowing us to define new support functions and Steiner points that are well behaved under isometries, but drop additivity. It is the case of the following definition, inspired by the aforementioned construction in loc. cit.

\begin{definition}
	Let $f: \bb R \to \bb R_{\geq 0}$ be a non-negative function such that $f^{-1}(0)=\{0\}$. The \textbf{exponential support} with parameter $f$ of a compact, tame shape $M \subseteq \bb R^n$ is the function \linebreak${\tilde h_M: \bb R^n \to \bb R}$ defined by
	\[ \tilde h_M(x) :=\ln \left[ \int_{\bb R} e^{-t} f(\ECT(M)(-x,t)) dt \right] \, . \]
\end{definition}

\begin{prop}
	The exponential support extends the support function for compact convex sets, that is, for a compact, convex set $K$:
	\[ \tilde h_K = \ln f(1) + h_K.\]
\end{prop}

\begin{proof}
	In a similar fashion as in the proof of lemma \ref{SuppMatchesConvex}, for $K$ convex, the integral equals
	\[ \int_{-h_K(x)}^{+\infty} e^{-t} f(1) dt= f(1) e^{h_K(x)},\]
	from which the result follows.
\end{proof}
So, when $f(1)=1$, this is the support function on convex sets: we consider such an $f$ for the remainder of the article. Reasonable choices might include the absolute value or square function.

We can now define Steiner points and Steiner supports based on this new extension of the support function. Remarkably, there is now no need to divide by the Euler characteristic:

\begin{definition}
	The \textbf{exponential Steiner point} of a compact, tame shape $M\subseteq \bb R^n$ is defined by
	
	\[ \tilde s(M)=\frac{1}{\mu(\bb B^n)} \int_{\bb S^{n-1}} v \cdot \tilde h_M(v) dv \, .\]
\end{definition}

\begin{theorem}
	The exponential Steiner point is covariant under euclidean isometries:
	\[ \tilde s(T(M))=T(\tilde s(M)) \]
	for every euclidean isometry $T=L+a$.
\end{theorem}

\begin{proof}
	One checks that $\tilde h_{L(M)+a}(v)=\tilde h_{M}(L^{-1}(v))+\langle v,a \rangle$ via the change of variables $\tilde t = t + \langle v,a \rangle$ (see the proof of proposition \ref{IsoSupport}) and properties of the natural logarithm. The result then follows from the value of $\int_{\bb S^{n-1}} v \cdot \langle v,a \rangle \: dv$ being $\mu (\bb B^n) a$.
\end{proof}

\begin{definition}
	\label{DefExpSupp}
	The \textbf{exponential Steiner support} $\tilde S_M$ of a compact, tame shape $M$ is defined as the exponential support of $M-\tilde s(M)$, that is, as $\tilde S_M(v):=\tilde h_M(v) -\langle v,\tilde s(M) \rangle$.
\end{definition}

Norms and distances stemming from the exponential Steiner support can now be defined as in definitions \ref{DefNormSupp} and \ref{DefDistSupp}, with the same behavior under isometries. Our experiments seem to indicate that exponential Steiner supports are more discriminative than the usual Steiner support, which we believe is due to their non-additivity (see section \ref{SectExperiment}).
	
	\section{Geometric formulas and bounds in the plane}
\label{SectGeo}
The formulas relating the support function to the perimeter and the area in the plane (see section \ref{convex}) allow us to produce new geometric bounds for ECT-induced norms and distances.

\subsection{The perimeter recovery formula and perimetric bounds}
We begin by noting that the edges of every simplicial complex in the plane bound either two, one, or no 2-dimensional faces. We call these edges \textbf{internal, boundary} and \textbf{isolated} edges, respectively (see figure \ref{FigSimpComp}).

\begin{definition}
	Let $S\subseteq \bb R^2$ be a simplicial complex. We define its $\textbf{perimeter}$ $\Peri(S)$ as the sum of the lengths of all boundary edges, plus two times the sum of the lengths of isolated edges.
\end{definition}

\begin{figure}
	\centering
	\includegraphics[scale = 0.6]{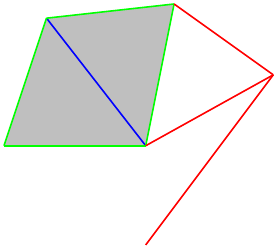}
	\caption{A flat simplicial complex. The internal edge is marked in blue, the boundary edges in green, and the isolated edges in red.}
	\label{FigSimpComp}
\end{figure}

Notice that, under this definition, the perimeter of a line segment is twice its length. This may seem strange, but picture a shrinking rectangle collapsing to a line: it makes sense that the perimeter remains continuous at the limit. One can picture lines as ``bigons'' under this perspective. 

We now extend this definition to any flat, compact, tame set $M$. We make use of the fact that any such set admits a $\mathcal C^1$ triangulation, that is, a $\mathcal C^1$ homeomorphism $f: S \to M$, where $S$ is a simplicial complex (see, for instance, \cite{pawlucki_strict_2024}).

\begin{definition}
	Let $M \subseteq \bb R^2$ be compact and tame, and let $f: S \to M$ be a $\mathcal C^1$ triangulation. We define the perimeter of $M$ as
	\[ \Peri(M):=\sum_{e\in \mathcal B(S)} \ell(f(e)) + 2\sum_{e \in \mathcal I(S)} \ell(f(e)) \, ,\]
	where $\ell$ denotes the usual length of a $\mathcal C^1$ curve, and $\mathcal B(S)$ and $\mathcal I(S)$ are the sets of boundary and isolated edges of $S$ respectively.
\end{definition}

It is easy to see that this definition does not depend on the chosen triangulation by invariance of length under reparametrization. We can now state the main result of this section, a rephrasing of Crofton's theorem in the plane (see, for instance, \cite{hug_integral_2020}) under the ECT language:

\begin{theorem}[Perimeter recovery]
	\label{PeriRecovery}
	Let $M \subseteq \bb R^2$ be compact and tame, and let $a \in \bb R^2$. Then:
	\[ \int_{\bb S^1} \int_{\bb R} (\ECT(M)(v,t)- \chi(M) \ECT(a)(v,t)) dt dv=\Peri(M)\, .\]
	In particular, \[ \int_{\bb S^{n-1}} h_M(v) dv=\int_{\bb S^{n-1}} S_M(v) dv=\Peri(M)\, . \]
\end{theorem}

\begin{proof}
	We first prove the statement in the case where $M$ is a simplicial complex. By replacing $M$ by $M-a$, we may assume without loss of generality that $a$ is set at the origin, so we compute $\int_{\bb S^{n-1}} h_M(v) dv$.
	The formula holds for any simplex of dimension 2 or lower in the plane, by virtue of the convex case --see equation (\ref{PeriSupport})-- or direct computation. The goal now is to express $M$ as a union of some of its simplices and apply general additivity of the support function.
	
	Now, any flat simplicial complex may be written as a union of all its faces, isolated edges (both taken with their boundaries), and isolated points. All twofold intersections between members of this family are either vertices --which we may ignore since the value of the resulting integral will be zero-- or internal edges. The perimeter of each internal edge is therefore subtracted and thus will cancel with the side lengths provided by both faces each one bounds. All threefold and higher intersections can only be vertices, so again, we ignore them. Only the perimeters of isolated edges and lengths of boundary edges remain, and so the statement follows. 
	
	In the general case, theorem \ref{Approx} in the appendix shows that we can pick a sequence $\{S_n\}_{n \in \bb N}$ of simplicial complexes with $\Peri(S_n) \overset{n \to \infty}\rightarrow \Peri(M)$ and $\int_{\bb S^{n-1}} h_{S_n}(v) dv \overset{n \to \infty}\rightarrow \int_{\bb S^{n-1}} h_M(v) dv $. We have just proven that both sequences are the same, and so their limits are equal.
\end{proof}

\begin{coro}
	\label{NormGeoIneqs}
	Let $M\subseteq \bb R^2$ be compact and tame. Then,
	\[ \Peri(M) \leq \normsupp M_1 \]
	and
	\[ \Peri(M) \leq \normcent M_{1,1}\, . \]
	 If, in addition, $S_M(v) \geq 0$ for all $v \in \bb S^{n-1}$, then:
	\[\normsupp M_1=\Peri(M)\, .\] 
	If, furthermore, $\overline{\ECT}(v,t) \geq 0$ for all  $v \in \bb S^{n-1}$ and $t \in \bb R$, then
	\[ \normcent M_{1,1}=\Peri(M)\, . \]
	\qed
\end{coro}

\begin{theorem}[Perimeter lower bound]
	\label{PeriLowBound}
	Let $M, N \subseteq \bb R^2$ be compact and tame shapes. Then:
	\[ |\Peri(M)-\Peri(N)|\leq d_1(M,N)\, ,  \]
	where $d_1$ is any of $\dsupp_1$, $\dcent_{1,1}$, or, if $M$ and $N$ have the same Euler characteristic, $\dist{1}{1}$.
\end{theorem}

\begin{proof}
	Writing $S_M(v)= \int_{\bb R} \overline{\ECT}(M)(v,t) dt$, theorem \ref{PeriRecovery} provides
	\[ |\Peri(M)-\Peri(N)|=\left| \int_{\bb S^1} \int_{\bb R} (\overline{\ECT}(M)(v,t)- \overline{\ECT}(N)(v,t) )dt dv\right|. \]
	The obvious inequalities obtained by introducing the absolute value inside either integral settle the result for the support and Steiner centered distance.\footnote{For the case of shapes with Euler characteristic zero, this equality holds with the alternative definition of $\overline{\ECT}$ of remark \ref{RmkAlternativeCentered}, since the double integral of $\ECT(a)(-v,t)-\ECT(\chi(M)s(M))(-v,t)$ vanishes. Employing this alternate definition settles the bound for the support distance in case any shape has Euler characteristic zero.} For the classical ECT-induced distance for shapes with equal Euler characteristic, we have, for an arbitrary $a \in \bb R^2$, \[\ECT(M)-\ECT(N)=(\ECT(M)-\chi(M)\ECT(a))-(\ECT(N)-\chi(N)\ECT(a)) \, ,\]
	so that
	\[ |\Peri(M)-\Peri(N)| =\left| \int_{\bb S^1} \int_{\bb R} ({\ECT}(M)(v,t)- {\ECT}(N)(v,t) )dt dv\right| \, ,\]
	and the same argument as before finishes the proof.
\end{proof}

The preceding corollary is particularly important, as it ensures that if two shapes have distinct perimeter, their support distance will be nonzero. As such, in the plane, the support distance is only at risk of misidentifying two distinct shapes if they have equal perimeter. Similar lower bounds can be obtained for other $p,q$ parameters by employing classic inequalities for $p$-norms.

Upper bounds for these distances are harder to come by, but there is a particularly nice result for simply connected shapes. We need a technical lemma first:
\begin{lemma}
	\label{SimpCompPositive}
	Let $M \subseteq \bb R^2$ be simply connected, compact and tame, and let $a$ be any point lying in its convex closure. Then,
	\[ \ECT(M)(v,t)-\ECT(a)(v,t) \geq 0 \]
	for all $v \in \bb S^1$ and $t \in \bb R$.
\end{lemma}

\begin{proof}
First, we prove that $\ECT(M)(v,t)$ never drops below $1$ whenever $M_{v, t}\neq \emptyset$. Since the Euler characteristic of a flat shape equals its number of connected components minus the dimension of its first homology group, the only way for the Euler characteristic of a sublevel to drop to zero is if $M_{v,t}$ has non-trivial one-dimensional homology. Write $M=M_{v,t} \cup M_{-v,-t}$. The Mayer-Vietoris exact sequence for $M_{v,t}$ and $M_{-v,-t}$ reads:
\[ \cdots \to H_1(M_{v,t} \cap M_{-v,-t})\to H_1(M_{v,t}) \oplus H_1(M_{-v,-t})\to H_1(M)\to \cdots \]
The first group is zero since $M_{v,t} \cap M_{-v,-t}$ is contained in a line. The last group is zero by hypothesis. By exactness, $H_1(M_{v,t})$ must be zero. Thus, $\ECT(M)(v,t)\geq 1\geq \ECT(a)(v,t)$ whenever $M_{v,t}\neq \emptyset$.

Finally, since $a$ lies in the convex closure of $M$, $\ECT(a)(v,t)=0$ whenever $M_{v,t}= \emptyset$, so in the case $M_{v,t}=\emptyset$, the inequality holds with both sides equal to $0$.
\end{proof}

\begin{theorem}[Perimeter upper bound for simply connected shapes]
	\label{PeriUpper}
	Let $M$ and $N$ be two flat, simply connected, compact and tame shapes. Then,
	\[ \dist{1}{1}(M,N)\leq \Peri(M)+ \Peri(N) +4 \dd(\operatorname{cl}(M), \operatorname{cl}(N)), \]
	where $\dd$ denotes the usual euclidean distance.
	
	If, in addition, each Steiner point of $M$ and $N$ lies inside the convex closure of $M$ and $N$ respectively,
	\[ \dsupp_1(M,N)\leq \Peri(M)+\Peri(N)\, , \]
	\[ \dcent_{1,1}(M,N) \leq \Peri(M)+\Peri(N) \, . \]
\end{theorem}

\begin{proof}
	For the latter two inequalities, the hypotheses in combination with lemma \ref{SimpCompPositive} ensure that $\overline{\ECT}(M)(v,t)$ and $\overline{\ECT}(N)(v,t)$ are non negative, so corollary \ref{NormGeoIneqs} ensures that the support and Steiner centered norms of $M$ and $N$ are equal to their respective perimeters. The result then follows from the distance between two shapes being less than or equal to the sum of their norms.
	
	For the first inequality, let $a \in \operatorname{cl}(M)$ and $b \in \operatorname{cl}(N)$ be arbitrary. The triangle inequality yields
	\[ \dist{1}{1}(M,N)\leq \dist{1}{1}(M,a)+\dist{1}{1}(a,b)+\dist{1}{1}(b,N)\, .\]
	Lemma \ref{SimpCompPositive} allows us to drop the absolute values when computing the first and last distances, and so they are equal to the perimeter of $M$ and $N$ by theorem \ref{PeriRecovery}. A computation shows that in the plane, the $(1,1)$-ECT-induced distance between two points is four times their euclidean distance. Taking the infimum over points in the convex closures of $M$ and $N$ yields the result.
\end{proof}

\subsection{A generalization of Barbier's theorem}

Recall from theorem \ref{meanwidth} and the discussion beforehand that, in any dimension, $h_K(v)+h_K(-v)$ is the width of a convex set $K$ when seen from direction $v$, and so $\frac{2}{\mu(\bb S^{n-1})}\int_{\bb S^{n-1}} h_K(v) dv$ is its mean width. These quantities lack a direct geometrical interpretation when $K$ is not convex\footnote{Although, see definition \ref{DefGenRadius}.}, but we can nonetheless extend the definitions and state a small generalization of Barbier's theorem.

\begin{definition}
	Given a compact, tame shape $M$, we define its \textbf{support width} in direction $v$ as
	\[ w_M(v):=h_M(v)+h_M(-v) \, .\]
	Its \textbf{mean support width} is, then, given by the integral
	\[ \overline{W}(M):= \frac{2}{\mu(\bb S^{n-1})} \int_{\bb S^{n-1}} h_M(v) dv \, .\]
	A shape $M$ is said to have \textbf{constant support width} if $w(v)$ is a constant function.
\end{definition}

\begin{theorem}[Generalized Barbier's theorem]
	Let $M\subseteq \bb R^2$ be a compact, tame shape. Then,
	\[ \Peri(M)=\pi \overline W(M). \]
	In particular, all flat shapes of constant support width $w$ have the same perimeter, $\pi w$.
	\label{BarbierGen}
\end{theorem}

\begin{proof}
	Follows immediately from the definition of mean support width and theorem \ref{PeriRecovery}.
\end{proof}

As an example, see figure \ref{FigHollowedCircle}: a circle for which an arc has been inwardly reflected along a chord. Explicitly computing its support function shows that it has constant support width, equal to two times its diameter, and its perimeter is indeed equal to $\pi$ times its diameter, counted twice (recall that for 1-dimensional shapes, the perimeter was defined as twice their length). Other trivial examples may be made by considering disjoint Reuleaux polygons or by picking a Reuleaux polygon and peppering it with Reuleaux polygon-shaped holes.

\begin{figure}
	\centering
	\includegraphics[scale=0.2]{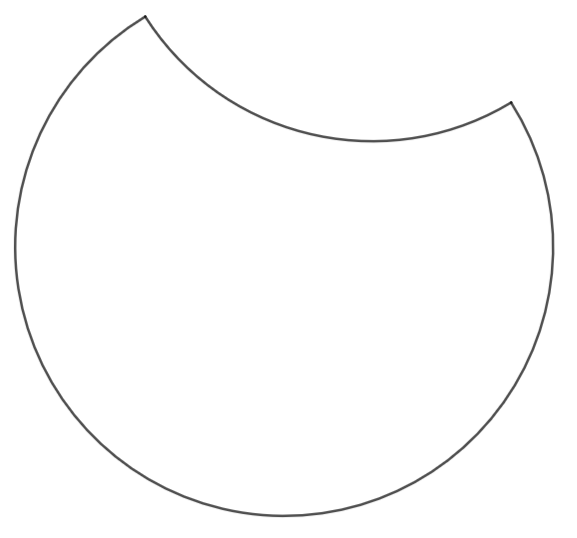}
	\caption{A shape of constant support width.}
	\label{FigHollowedCircle}
\end{figure}

\subsection{Area and curvature}

We can also relate the area and external angles of convex polygons to their ECT. Moreover, this provides another interesting property of the Steiner point.

\begin{theorem}
	\label{Area}
	Let $K \subseteq \bb R^2$ be a compact convex set. Suppose that for each $v\in \bb S^{n-1}$ there is a unique ``last-hit'' point $x(v)\in K$, that is, such that $h_K(v)=\langle x(v), v \rangle$. Suppose also that $h_K$ is twice differentiable and $K$ has a well defined radius of curvature $\rho(v)$ at each $x(v)$. Then, for any $a \in \bb R^2$:
	\[ \int_{\bb S^{1}} \left(\int_{\bb R} (\ECT(K)(v,t)-\ECT(a)(v,t)) dt \right)^2 dv= \operatorname{Area}(K)+\frac{1}{2}\int_{\bb S^{1}} \|x(v)-a\|^2  dv \, .\]
	In the case where $K$ is a convex polygon, label $\alpha_1,...,\alpha_n$ the exterior angles of $K$ and $d_1,...,d_n$ the distances from $a$ to each vertex. Then:
	\[ \int_{\bb S^{1}} \left(\int_{\bb R} (\ECT(K)(v,t)-\ECT(a)(v,t)) dt \right)^2 dv=\operatorname{Area}(K)+\frac{1}{2}\sum_{i=0}^n \alpha_i d_i^2\, .  \]
	The unique point $a\in \bb R^2$ minimizing these values is $s(K)$, the Steiner point of $K$.
\end{theorem}

\begin{proof}
	Replace $K$ by $K-a$ so that the left hand side is $\int_{\bb S^{1}} h_K(v)^2 dv$.
	Theorem \ref{SupportArea} provides the following equalities:
	\[ \operatorname{Area}(K)=\frac{1}{2} \int_{\bb S^{1}} h_K(v) \rho(v) dv\, ,\]
	and
	\[ \rho(v)=h_K(v)+h''_K(v)\, .\]
	Substitution and integration by parts shows that the area of $K$ may be written as
	\begin{align*}
		\operatorname{Area}(K)&=\frac{1}{2}\left( \int_{\bb S^{1}} h_K(v)^2 dv +\int_{\bb S^{1}} h_K(v) h_K''(v) dv \right) \\
		&=\frac{1}{2}\left( \int_{\bb S^{1}} h_K(v)^2 dv -\int_{\bb S^{1}} h_K'(v)^2 dv \right).
	\end{align*}
	If we show $\int_{\bb S^{1}} \|x(v)\|^2 dv = \int_{\bb S^{1}} h_K(v)^2 dv +\int_{\bb S^{1}} h_K'(v)^2 dv  $, the first formula will be proven.
	
	 It is a general fact (see \cite{schneider2013convex}, corollary 1.7.3) that the gradient of $h_K$ evaluated at $v$ is equal to $x(v)$. Evaluating $h'_K$ on the circle amounts to taking the directional derivative of $h_K(v)$ in direction $v^\perp$, the orthogonal of $v$ (with a counter-clockwise 90º turn). As such,
	
	\[ \int_{\bb S^{1}} h_K(v)^2 dv +\int_{\bb S^{1}} h_K'(v)^2 dv=\int_{\bb S^{1}} \langle x(v),v \rangle^2 +\langle x(v), v^\perp \rangle^2 dv=\int_{\bb S^1} \|x(v)\|^2 dv\, .\]
	
	In the case of a convex polygon, the formula can be proven directly by a somewhat tedious computation. Instead, we remark how it follows from the formula just proven even if the hypotheses are not met: indeed, the last point hit at a direction $v$ is a certain vertex of the polygon, and it is so for a whole range of directions with size equal to the exterior angle of the vertex.
	
	The statement about the Steiner point is proposition \ref{SteinerMinimizer}.
\end{proof}

\begin{coro}
	Under the same hypotheses of theorem \ref{Area},
	\[ \normsupp {K}_2 = \operatorname{Area}(K)+\frac{1}{2}\int_{\bb S^{n-1}} \|x(v)-s(K)\|^2 dv \leq \operatorname{Area}(K) +\pi \max_{x \in \partial K} \|x-s(K)\|\, .\] \qed
\end{coro}

\section{Isometry invariant features and pseudodistances}
\label{SectIsometry}
\subsection{Steiner barcodes}
As we have proven in proposition \ref{LinEquiv}, the Steiner support is a translation invariant feature of a shape. As such, the pseudodistance stemming from it is also invariant under translations. The next natural step would be to define similar features and pseudodistances that are invariant under all isometries, that is, that allow shapes to be compared even when arbitrary isometries are applied to any of them. Proposition \ref{LinEquiv} hints at a way to do so: the Steiner support of a shape that has been, for instance, rotated, is nothing but the original Steiner support, but evaluated after applying the inverse of the rotation. As such, if we picture the Steiner support $S_M: \bb S^{n-1} \to \bb R$ as a surface lying over the sphere, the overall shape of the surface is unchanged, in particular, its level sets are isomorphic. As such, an analysis of the topology of its sublevel sets, be it via persistent homology or an Euler curve, will produce a rotation invariant feature. Here, we have chosen the persistent homology approach as it is the most informative of the two, but any persistence-based calculation on sublevel sets should do. For readers unacquainted with persistence and its role in TDA, we refer them to \cite{oudot2015persistence} for an introduction.

\begin{definition}
	\label{DefBarcodes}
	Let $M$ be a compact, tame shape. The $i$-th \textbf{Steiner barcode} of $M$ is defined as the $i$-th persistent homology barcode of the following filtration of $\bb S^{n-1}$:
	$$\{ v \in \bb S^{n-1} | S_M(v) \leq t \}_{t \in \bb R} \, .$$
	
	Given two compact, tame shapes $M$ and $N$, their $i$-th \textbf{bottleneck pseudodistance} $\dbottle(M,N)$ is the value of the bottleneck distance between their respective $i$-th Steiner barcodes.
\end{definition}

\begin{theorem}[Isometry invariance of Steiner barcodes]
	\label{BarcodesInvariance}
	Steiner barcodes of a shape are invariant under euclidean isometries. Bottleneck pseudodistances between two shapes remain invariant under euclidean isometry of any of the involved shapes.
\end{theorem}

\begin{proof}
	Given an isometry $T$, the sublevel sets of $S_M(v)$ and of $S_{T(M)}(v)$ are related under a fixed automorphism $L^{-1}$ of $\bb S^{n-1}$ by proposition \ref{LinEquiv}, and hence $L^{-1}$ provides an isomorphism between filtrations. This descends to an isomorphism at the persistent homology level, and the assertions follow.
\end{proof}

As an immediate consequence of bottleneck stability, we have the following bound:
\begin{theorem}
	Let $M$ and $N$ be compact, tame shapes. Then:
	\[ \dbottle(M,N) \leq \inf_{T_1, T_2} \dsupp_\infty(T_1(M),T_2(N))\, ,  \]
	where $T_1$ and $T_2$ range over all euclidean isometries.
\end{theorem}
\begin{proof}
	This follows from the general bound for the bottleneck distance between diagrams stemming from sublevel sets of continuous functions (see for instance corollary 3.6 of \cite{oudot2015persistence}) and theorem \ref{BarcodesInvariance}.
\end{proof}

Of course, all of this can be defined with either the usual or the exponential Steiner support (definition \ref{DefExpSupp}). Wasserstein distances may be taken instead of the more usual bottleneck.

\subsection{The generalized radius}
We mention an isometry invariant, real-valued feature of a shape that may be of geometric interest:
\begin{definition}
	\label{DefGenRadius}
	We call \textbf{generalized radius} of a shape $M$ the following quantity:
	\[ \frac{1}{\mu(\bb S^{n-1})} \int_{\bb S^{n-1}} S_M(v) dv.\]
\end{definition}

\begin{remark}
	One can see that this is equal to $\frac{1}{\mu(\bb S^{n-1})} \int_{\bb S^{n-1}} h_M(v) dv$, so there is actually no need of computing the Steiner point and Steiner support. In $\bb R^2$, it equals the perimeter of a shape divided by $2\pi$, by theorem \ref{PeriRecovery}.
\end{remark}

The reason for the name is the following:
\begin{prop}
	When positive, the generalized radius of a shape $M\subseteq \bb R^n$ is the radius of the $n$-ball which is at minimum 2-support distance from $M$.
\end{prop}

\begin{proof}
	The Steiner support of a ball $\bb B_R$ of radius $R$ is the constant function on the sphere with value $R$. Thus:
	\[ \dsupp_2(M, \bb B_R)^2=\int_{\bb S^{n-1}}\left( S_M(v)-R \right)^2 dv.\]
	The best constant approximation to $S_M(v)$ in $L^2$ norm is given by its mean value, precisely the definition of the generalized radius.
\end{proof}

The generalized radius may be more amenable to interpretation than Steiner barcodes, as it relates it directly with the radius of a ball, although it is a much less informative feature. 

\section{Implementation and experimental results}
\label{SectExperiment}
We have run some experiments to assess if Steiner supports and barcodes can be useful in shape classification tasks. We summarize our results in this section.
\subsection{MPEG-7 dataset}
\paragraph{Procedure and computational implementation}  For our first experiments, we have picked the MPEG-7 dataset\cite{MPEG}, as several papers in topological data analysis use it as a case study. In particular, we closely followed the experiments run in the recent preprint \cite{yang_topological_2026} that also proposes ECT-based isometry invariant features and uses the MPEG-7 image dataset as a benchmark for ECT-based classification methods. We refer in particular to table S1 in loc. cit. for an overview.

MPEG-7 consists of 1400 black and white pictures, organized in 70 categories of 20 similar looking images each. The goal, either with a selected 10-class subset of the MPEG-7 dataset\footnote{Consisting of all images from the classes ``apple'', ``bell'', ``bottle'', ``car'', ``classic'', ```cup'', ``device0'', ``face'', ``Heart'', and ``key''.}, or the whole 70-class set, is to train a Support Vector Machine (SVM) to classify images in their correct category. We will do so by extracting their Steiner supports and barcodes.

Parameters in our implementation were taken as closely as possible to the ones in \cite{yang_topological_2026}. However, comparing the reported accuracies would be misleading: for ease of implementation, we do not work with the raw black-and-white image data, preprocessing them instead by extracting the outlines of the shapes, which improves accuracy over binary image data. We have also run training with raw ECT data, as can be seen in the last rows of table \ref{AccTable}, to more fairly compare our approach with a classical one. Our preprocessing differs from the one taken in loc. cit., and indeed, ECT data is achieving higher accuracy values for us than the ones reported therein (compare, for instance, the value 98.69 \% in the first column of table \ref{AccTable}, with the reported 86.23 \%). In summary, higher accuracies than those of loc. cit. are to be expected and are not indicative of better performance.

To compute ECTs, we have opted for using the \texttt{ect} python package \cite{ayub_ect_2026}. For barcode computation, we use the \texttt{GUDHI} package \cite{gudhi:urm}.
The pipeline we have implemented is as follows:

\begin{enumerate}
	\item The images are preprocessed by extracting their outlines as arrays of two dimensional points, and picking one out of five points in each outline to reduce their size and jaggedness. Some padding is added beforehand to ensure the outlines are computed correctly.
	\item Making use of the \texttt{ect} package, the outlines of each image are stored as a complex consisting of cycles, the coordinates are centered and normalized to a radius of 1, and their ECTs are computed on 100 evenly spaced directions and on 100 height values.
	\item For each complex, its Steiner support is computed by numerical integration. Its zeroth Steiner barcode is computed in \texttt{GUDHI} by first constructing a simplicial model of the circle whose vertices are the 100 evenly spaced directions the Steiner support is computed at, setting the birth time of each vertex $v$ to be $S_M(v)$, and finally computing the zero-dimensional persistent homology of the resulting filtered simplicial complex.
	\item Each Steiner support is vectorized by listing all 100 values taken along the circle (starting at the point $(-1,0)$ and advancing counterclockwise). Each barcode is vectorized as its first three \emph{persistence landscapes}, a common vectorization technique for barcodes (see \cite{bubenik2020persistence} for an introduction). Each of the three landscapes consists of a vector of length 100.
	\item A random 70/30 (resp. 90/10) train/test split is done on the 10 classes (resp. 70 classes) data, and a SVM is trained with a linear kernel with the training data. Prediction is done on the test data and accuracy is computed.
	\item Step 5 is repeated 500 (resp. 10) times with random splits, and average accuracy is computed.
\end{enumerate}

We have run this experiment with either the usual Steiner support, the exponential Steiner support\footnote{The parameter function in these experiments was chosen to be $f(x)=x^2$.}, or both, and either providing the SVM with the support data, landscapes, or both. Finally, we also train SVMs on the whole ECT data we have extracted for support computation, each vectorized as a vector of length 10,000. We also repeat the ECT training with a reduced version of ECT data, computed only at 40 directions and 20 height values each, so as to match our least lightweight method. We report in the first columns of table \ref{AccTable} the average accuracy for each test and method.

\begin{table}[]
	\centering
	\begin{tabular}{l|l|l|l|l|l|}
		\cline{2-6}
		& \begin{tabular}[c]{@{}l@{}}Feature\\ dimension\end{tabular} & 10 classes        & 70 classes        & \begin{tabular}[c]{@{}l@{}}10 classes\\ (misaligned)\end{tabular} & \begin{tabular}[c]{@{}l@{}}70 classes\\ (misaligned)\end{tabular} \\ \hline
		\multicolumn{1}{|l|}{Steiner Support}              & 100                                                         & 94.99 \%          & 53.71 \%          & 59.00 \%                                                          & 12.86 \%                                                          \\ \hline
		\multicolumn{1}{|l|}{Steiner Landscapes}           & 300                                                         & 81.63 \%          & 32.07 \%          & 81.58 \%                                                          & 31.36 \%                                                          \\ \hline
		\multicolumn{1}{|l|}{Steiner Support + Landscapes} & 400                                                         & 94.97 \%          & 59.79 \%          & 84.87 \%                                                          & 34.43 \%                                                          \\ \hhline{|=|=|=|=|=|=|}
		\multicolumn{1}{|l|}{Exponential Support}          & 100                                                         & 94.56 \%          & 63.07 \%          & 65.34 \%                                                          & 16.43 \%                                                          \\ \hline
		\multicolumn{1}{|l|}{Exp. Landscapes}              & 300                                                         & 73.73 \%          & 36.43 \%          & 71.06 \%                                                          & 34.43 \%                                                          \\ \hline
		\multicolumn{1}{|l|}{Exp. Supp. + Landscapes}      & 400                                                         & 95.50 \%          & 69.57 \%          & 75.01 \%                                                          & 36.14 \%                                                          \\ \hhline{|=|=|=|=|=|=|}
		\multicolumn{1}{|l|}{Both Supports}                & 200                                                         & 97.09 \%          & 73.93 \%          & 78.45 \%                                                          & 33.21 \%                                                          \\ \hline
		\multicolumn{1}{|l|}{Both Landscapes}              & 600                                                         & 93.71 \%          & 57.21 \%          & \textit{\textbf{95.05 \%}}                                        & 55.43 \%                                                          \\ \hline
		\multicolumn{1}{|l|}{Both Supports + Landscapes}   & 800                                                         & \textit{97.06 \%} & \textit{78.43 \%} & 91.62 \%                                                          & \textit{\textbf{56.86 \%}}                                        \\ \hhline{|=|=|=|=|=|=|}
		\multicolumn{1}{|l|}{Raw ECT data}                 & 10,000                                                      & \textbf{98.69 \%} & \textbf{85.86 \%} & 79.06 \%                                                          & 45.86 \%                                                          \\ \hline
		\multicolumn{1}{|l|}{Raw ECT data (reduced)}       & 800                                                         & 98.50 \%          & 84.79 \%          & 77.58 \%                                                          & 43.36 \%                                                          \\ \hline
	\end{tabular}
	\caption{Average accuracy of SVMs trained with either support data, landscapes, or both, of either a 10 class subset of the MPEG-7 dataset or the full 70 class dataset. Also compared with training over the raw ECT data. The first two columns report the results on the original MPEG-7 dataset, and the latter two with a randomly misaligned version of it. The highest accuracy achieved for each experiment is marked in bold. The highest accuracy among our proposed methods is marked in italics. }
	\label{AccTable}
\end{table}

\paragraph{Results and discussion} On the 10 classes test, our most lightweight methods --Steiner supports and exponential supports-- achieve an accuracy of over 94 \%, while suffering a severe drop-off in the 70 classes test. Comparing Steiner supports with exponential supports, the latter seem to be more accurate, in line with our intuition of them being more discriminative thanks to their non-additivity. Moreover, training on data provided by both supports improves accuracy by quite a margin, which is particularly apparent on the 70 class data. Thus, information provided by both supports does not seem to be redundant.

Training the SVM purely on the barcodes' landscapes yields worse results across the board: indeed, they are the least informative. Combining supports with their barcodes generally boosts accuracy, albeit by a small margin. Still, our best methods are outperformed in both the 10 classes and 70 classes test by raw ECT data, which is more informative. Indeed: support functions aggregate much of the information of the ECT, and thus are prone to information loss.

However, an important remark is that MPEG-7's images are mostly prealigned, meaning that images in the same set usually are oriented consistently. This has allowed us to utilize the Steiner support functions and the ECT as input to the SVM, which are not rotation invariant, to get good results. For misaligned datasets, we expect to see a significant drop-off in accuracy for ECT and support-based methods, but not for barcode-based methods, since Steiner barcodes are isometry invariant. To test this, we have produced a new dataset consisting of random rotations of the contours of MPEG-7 images and tested our methods in this misaligned dataset. Average accuracy evaluations can be found in the last two columns of table \ref{AccTable}. As expected, support-based methods perform much worse on the misaligned dataset, and are outperformed by landscapes. The accuracy of purely landscape-based methods has not seen a drop-off, as was to be expected: small deviations are seen, but this is likely due to imprecisions in our implementation. The combined landscape data works best for this dataset, outperforming raw ECT data. The accuracy gain is not so apparent on the 70 class data, as, again, information loss is a bigger factor in large datasets.

An important remark to keep in mind is that in our preprocessing, we encoded each image as a collection of contours, essentially seeing the images as ``hollow''. The addition of higher-dimensional information to the dataset could have a notable effect on accuracy. 

\paragraph{Addition of noise}

To illustrate theorem \ref{ThmPointCloud} (robustness of the Steiner support under point cloud noise) and to show how Steiner supports could be better suited than the ECT for certain noisy datasets, we introduce noise in the form of point clouds to the shape outlines. We generate, at random, between 1 and 200 points for each shape and add them to the complex before computing their ECTs and support functions. Average accuracies are reported in table \ref{AccNoisy}. We do not report results for the exponential support: it is highly sensitive to this kind of noise and achieves very poor accuracy. Results show that ECT accuracy falls dramatically when noise is added, but accuracy of our proposed Steiner support methods remains roughly the same, as expected.
 
 \begin{table}[]
 	\centering
 	\begin{tabular}{l|l|l|l|l|l|}
 		\cline{2-6}
 		& \begin{tabular}[c]{@{}l@{}}Feature\\ dimension\end{tabular} & \begin{tabular}[c]{@{}l@{}}10 classes\\ (noisy)\end{tabular} & \begin{tabular}[c]{@{}l@{}}70 classes\\ (noisy)\end{tabular} & \begin{tabular}[c]{@{}l@{}}10 classes\\ (noisy, mis.)\end{tabular} & \begin{tabular}[c]{@{}l@{}}70 classes\\ (noisy, mis.)\end{tabular} \\ \hline
 		\multicolumn{1}{|l|}{Steiner Support}              & 100                                                         & 94.73 \%                                                     & 51.64 \%                                                     & 58.71 \%                                                           & 13.50 \%                                                           \\ \hline
 		\multicolumn{1}{|l|}{Steiner Landscapes}           & 300                                                         & 79,31 \%                                                     & 31.93 \%                                                     & 77.47 \%                                                           & 30.57 \%                                                           \\ \hline
 		\multicolumn{1}{|l|}{Steiner Support + Landscapes} & 400                                                         & \textit{\textbf{94.75 \%}}                                   & \textit{\textbf{58.86 \%}}                                   & \textit{\textbf{81.88 \%}}                                         & \textit{\textbf{32.43 \%}}                                         \\ \hhline{|=|=|=|=|=|=|}
 		\multicolumn{1}{|l|}{Raw ECT data}                 & 10,000                                                      & 51.94 \%                                                     & 27.57 \%                                                     & 17.65 \%                                                           & 3.79 \%                                                            \\ \hline
 		\multicolumn{1}{|l|}{Raw ECT data (reduced)}       & 800                                                         & 50.77 \%                                                     & 30.21 \%                                                     & 16.40 \%                                                           & 2.79 \%                                                            \\ \hline
 	\end{tabular}
 	\caption{Average accuracy of SVMs trained with either support data, landscape data, or both, of either a 10 class subset of the MPEG-7 dataset or the full 70 class dataset with added point cloud noise. Also compared with training over the raw ECT data. The first two columns report the results on the original MPEG-7 dataset, and the latter two with a randomly misaligned version of it. The highest accuracy achieved for each experiment is marked in bold. The highest accuracy among our proposed methods is marked in italics.}
 	 \label{AccNoisy}
 \end{table}

 \subsection{Triangles and Squares dataset}
 
 To further compare support-based methods with the traditional ECT, we have devised the following experiment. A dataset of simplicial complexes is generated, each consisting of either between 1 and 10 triangles, or between 1 and 10 squares (with their interiors) of varying sizes\footnote{Specifically, the radii of the inscribing circles of each polygon are randomized to be between 0.05 and 0.15 units long.}, distributed randomly within a disc. All triangles are consistently oriented, as well as all squares. The goal is to train SVMs on either ECT or support data to accomplish one of the two following tasks:
 \begin{itemize}
 	\item Predict the number of shapes (connected components) of each image. (Quantity test)
 	\item Classify the image as either consisting of triangles or of squares. (Type test)
 \end{itemize}
 Additionally, two levels of difficulty can be added when generating the dataset:
 \begin{itemize}
 	\item A random rotation is applied to each shape.
 	\item Each shape has a 50 \% chance of being ``hollowed out'', that is, consisting only of its boundary.
 \end{itemize}
 
 We generate 10 complexes for each shape quantity number between 1 and 10, for both triangles and squares, obtaining a dataset of size 200. ECT and support data for each complex is computed via the \texttt{ect} package. Random 70/30 train/test splits are done to train and test 100 SVMs on different kinds of data, for either the quantity task or the type task. We report mean accuracy results for both tests in table \ref{TriSqTable}.
 
 \begin{table}[]
 	\centering
 		\begin{tabular}{cl|l|l|l|l|}
 			\cline{3-6}
 			\multicolumn{1}{l}{}                                 &                            & Original          & \begin{tabular}[c]{@{}l@{}}Random\\ rotations\end{tabular} & \begin{tabular}[c]{@{}l@{}}Random\\ hollow\end{tabular} & \begin{tabular}[c]{@{}l@{}}Random\\ rot. + holl.\end{tabular} \\ \hline
 			\multicolumn{1}{|c|}{\multirow{7}{*}{Quantity test}} & Steiner Support            & 55.52 \%          & 62.68 \%                                                   & 41.93 \%                                                & 39.45 \%                                                      \\ \cline{2-6} 
 			\multicolumn{1}{|c|}{}                               & Steiner Supp. + Landscapes & 53.95 \%          & 63.33 \%                                                   & 42.90 \%                                                & 40.63 \%                                                      \\ \cline{2-6} 
 			\multicolumn{1}{|c|}{}                               & Exp. Support               & \textbf{100 \%}   & \textbf{100 \%}                                            & 28.48 \%                                                & 24.53 \%                                                      \\ \cline{2-6} 
 			\multicolumn{1}{|c|}{}                               & Exp. Supp. + Landscapes    & \textbf{100 \%}   & \textbf{100 \%}                                            & 35.58 \%                                                & 29.92 \%                                                      \\ \cline{2-6} 
 			\multicolumn{1}{|c|}{}                               & Both Supports              & \textbf{100 \%}   & \textbf{100 \%}                                            & 60.13 \%                                       & 62.30 \%                                                      \\ \cline{2-6} 
 			\multicolumn{1}{|c|}{}                               & Both Supp. + Landscapes    & \textbf{100 \%}   & \textbf{100 \%}                                            & \textbf{60.75} \%                                                & \textbf{63.37 \%}                                             \\ \cline{2-6} 
 			\multicolumn{1}{|c|}{}                               & Raw ECT data               & \textbf{100 \%}   & \textbf{100 \%}                                            & 33.23 \%                                                & 32.80 \%                                                       \\ \hhline{|=|=|=|=|=|=|}
 			\multicolumn{1}{|c|}{\multirow{7}{*}{Type test}}     & Steiner Support            & 95.78 \%          & 46.18 \%                                                   & 96.08 \%                                                & 47.78 \%                                                      \\ \cline{2-6} 
 			\multicolumn{1}{|c|}{}                               & Steiner Supp. + Landscapes & 95.97 \%          & \textbf{55.02 \%}                                          & 95.82 \%                                                & 48.78 \%                                                      \\ \cline{2-6} 
 			\multicolumn{1}{|c|}{}                               & Exp. Support               & 85.77 \%          & 43.80 \%                                                   & 92.77 \%                                                & 48.95 \%                                                      \\ \cline{2-6} 
 			\multicolumn{1}{|c|}{}                               & Exp. Supp. + Landscapes    & 85.60 \%          & 43.85 \%                                                   & 92.05 \%                                                & 49.38 \%                                                      \\ \cline{2-6} 
 			\multicolumn{1}{|c|}{}                               & Both Supports              & \textbf{98.33 \%} & 45.30 \%                                                   & \textbf{99.55 \%}                                       & 50.77 \%                                                      \\ \cline{2-6} 
 			\multicolumn{1}{|c|}{}                               & Both Supp. + Landscapes    & 97.58 \%          & 51.73 \%                                                   & 99.40 \%                                                & 50.88 \%                                                      \\ \cline{2-6} 
 			\multicolumn{1}{|c|}{}                               & Raw ECT data               & 86.68 \%          & 50.10 \%                                                   & 81.38 \%                                                & \textbf{51.83 \%}                                             \\ \hline
 		\end{tabular}
 	\caption{Mean accuracy results for the quantity and type tests in the Triangles and Squares dataset.}
 	\label{TriSqTable}
 \end{table}

\paragraph{Quantity test, discussion} In the quantity test, both ECT data and methods which incorporate exponential supports perform with perfect accuracy in the original dataset, as well as in the one with random rotations. For the ECT, this is not surprising: the number of connected components of the dataset is precisely equal to its Euler characteristic, since in these datasets all connected components are contractible. The exponential support is more sensitive than the Steiner support to the addition of connected components, and in this case this seems to be enough to obtain the number of connected components. The Steiner support performs the poorest in these datasets, likely because it is mostly blind to small shapes. Indeed, repeating experiments with a narrower range of sizes shows better performance of Steiner support across the board in the quantity test, which can indeed go up to {100 \%} for a small enough size deviation in the original and randomly rotated datasets. Exaggerating shape sizes yields, on the contrary, worse results. 

Addition of hollow shapes is challenging for Euler characteristic-based methods, since now there is no direct correlation between it and the number of connected components. The ECT and exponential support, who were the most reliable, see the most downfall of accuracy in this regard, to around 30 \%. The Steiner support seems to perform somewhat better on average, and combining both supports almost doubles the efficiency of raw ECT data. We speculate the following: since Steiner supports encapsulate some information about the perimeter of shapes, they should be more sensitive to hollowed shapes (whose perimeter is counted twice). On the other hand, exponential supports should be, on the contrary, more sensitive to non-hollowed shapes, since addition of a hollow shape only modifies each Euler curve in a small interval, while a non-hollow shape modifies it since the first time it is hit by the rising hyperplane and onward. Thus, each support provides a different kind of information from which the total number of shapes can be more easily inferred. Indeed, repeating the experiment with a narrower size range (0.1 to 0.15 units instead of 0.05 to 0.15) shows a dramatic improvement for the ``both supports'' method, jumping to an accuracy of 94.30 \%, while raw ECT data remains roughly the same, at 36.68 \% accuracy. Rotating the dataset in addition to hollowing out shapes does not appear to have a significant impact on accuracy on either of the methods for the quantity task.

\paragraph{Type test, discussion} In the original dataset, support-based methods, apart from those purely relying on the exponential support, perform with great accuracy and surpass the ECT. Surprisingly, with random shapes being hollowed out, the exponential support improves its performance, while other support methods remain roughly the same. This contrasts with ECT data, whose accuracy is somewhat lower when hollow shapes are added. No method is reliable in the type test with random rotations added, all accuracies being comparable to classification by random chance.
	
	\section{Beyond the Euler characteristic: a case for Betti supports}
\label{SectionBetti}
To motivate the definitions in this section, let us point out that the Steiner point, useful as it has been to us, has some quirks. Perhaps the most damning is that $s(M)$ may not lie in the shape $M$. In fact, it can be arbitrarily far from $M$, as the next example shows.

\begin{example}[Lollipop]
	\label{Lollipop}
	Consider the following ``lollipop'', the union of $A$, a circumference with a diameter, and $B$, a segment extending said diameter:
	\begin{figure}[h]
		\centering
		\includegraphics[scale = 0.6]{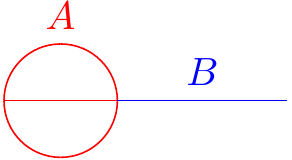}
		\caption{A lollipop $M$: the union of a circumference with a diameter ($A$) and a line segment ($B$).}
	\end{figure}
	
	To compute its Steiner point, we make use of the weighted additivity formula (proposition \ref{SteinerWeightAdd}) for $A$ and $B$. Both $A$ and $B$ are rotationally symmetric, so by the invariance of the Steiner point under isometries, $s(A)$ is the center of the circumference and $s(B)$ is the midpoint of $B$. Furthermore, $A\cap B$ is the unique point of intersection, so its Steiner point is itself. Taking into account the Euler characteristic of all involved shapes, the weighted additivity formula reads:
	\[ -s(A \cup B)+s(A \cap B)=-s(A)+s(B)\, , \]
	and so:
	\[ s(M)=s(A \cup B)=s(A)+s(A\cap B)-s(B)\, .  \]
	Now, $s(A \cap B)-s(B)$ is a vector from the midpoint of $B$ to its left-most end, and so to obtain $s(M)$, one starts at $s(A)$, the center of the circle, and travels to the left half the length of segment $B$. As long as $B$ is longer than the diameter of the circle, $s(M)$ will lie outside of $M$. The longer $B$ is, the further away from $M$ it will lie. 
\end{example}

From the minimizer point of view (see proposition \ref{SteinerMinimizer}), the behavior in this example can be blamed on the fact that the innermost integrand of the expression
\[  \int_{\bb S^{n-1}} \left( \int_{\bb R} (\ECT(M)(-v,t)-\chi(M)\ECT(a)(-v,t)) dt\right)^2 dv\] is free to change sign, and so a point seeking to minimize this value can sometimes exploit sign changes in Euler curves, placing itself far from the shape to allow cancellations to take place before adding its weight to the computation.

One way to solve this is to replace the Euler characteristic by non-negatively valued topological invariants: Betti numbers. This also comes with other benefits: the Euler characteristic coalesces all Betti numbers into a single, possibly negative, numerical invariant. By unraveling this information and studying shapes at the Betti numbers level we may improve descriptiveness and see all of our previous constructions in a new light: as agglomerations of homological information. This is, of course, heavily related to the Persistent Homology Transform (PHT) and Betti Curve Transform (BCT), more informative but more computationally demanding variants of the ECT \cite{turner2014persistent}. In the following, we will assume homology computations are done with coefficients in $\bb R$, so as to eliminate torsion issues.

\begin{definition}
	Let $M$ be a compact, tame shape. For each $v\in \bb S^{n-1}$, consider its sublevel filtration in direction $v$: $\{ M_{v,t} \}_{t \in \bb R}$. Let $\beta_i(v)(t)=\dim H_i(M_{v,t})$ be the $i$-th Betti curve of said filtration. The $i$-th \textbf{Betti support} of $M$ is the function $\beta_ih_M:\bb S^{n-1} \to \bb R$ given by
	\[\beta_ih_M(v):= \int_{\bb R} (\beta_i(-v)(t)-\beta_i(M)\bb 1_{[0, +\infty)}(t)) dt \, . \]
\end{definition}

\begin{coro}
	\label{SuppCombBetti}
	For any compact, tame shape $M$:
	\[ h_M(v)=\sum_{i}(-1)^i \beta_ih_M(v)\, . \]
	In particular, if $M$ is convex, $\beta_0h_M$ is its support function and all other Betti supports are zero.
\end{coro}

\begin{proof}
	The equality follows from $\chi(M_{v,t})=\sum_i (-1)^i \beta_i(v)(t)$ and $\chi(M)=\sum_i (-1)^i \beta_i(M)$. Any sublevel set of a convex set is convex, and so, only has nontrivial homology in dimension zero. Thus, $\beta_ih_M$ is identically zero for $i>0$, proving the assertion about convex sets.
\end{proof}

\begin{definition}
	Let $M$ be a compact, tame shape. For each dimension $i$, we define its $i$-th \textbf{Betti-Steiner point} $\beta_is(M)$ as the origin if $\beta_i(M)=0$, and as
	\[\beta_is(M):=\frac{1}{\beta_i(M)\mu(\bb B^n)} \int_{\bb S^{n-1}} v \cdot \beta_ih_M(v) dv\]
	otherwise. Equivalently, it can be defined as
	\[ \beta_is(M):= \arg\min_{a \in \bb R^n}\int_{\bb S^{n-1}} \left(\int_{\bb R} (\beta_i(-v)(t) -\beta_i(M)\bb 1_{[-\langle v,a \rangle, +\infty)}(t)) dt \right)^2 dv. \]
\end{definition}

As it was the case for the Euler characteristic, a shape without an $i$-dimensional Betti number does not have a Betti-Steiner point, but it does have an \textbf{unweighted Betti-Steiner point}
\[ \beta_i(M)\beta_is(M):=\frac{1}{\mu(\bb B^n)} \int_{\bb S^{n-1}} v \cdot \beta_ih_M(v) dv\, .\]
As for the Euler characteristic, whenever we write a Betti-Steiner point accompanied by its corresponding Betti number, we will mean this unweighted version, even if said Betti number is zero.
\begin{prop}
	Given a compact, tame shape $M$, if $\beta_i(M) \neq 0$, the $i$-th Betti-Steiner point of $M$ is covariant under euclidean isometries:
	\[ \beta_is(T(M))=T(\beta_is(M)) \]
	for all euclidean isometries $T$.
\end{prop}
\begin{proof}
	The proof of \ref{SteinerIsometry} holds verbatim, replacing supports by Betti supports and the Euler characteristic by the $i$-th Betti number.
\end{proof}

These points provide another way to think about the unusual behavior of the usual Steiner point: it is a sort of summary of all Betti-Steiner points:

\begin{coro}
	\label{StPointCombBetti}
	For any compact, tame shape, $M$:
	\[ \chi(M)s(M)=\sum_{i} (-1)^i \beta_i(M) \beta_is(M) \]
\end{coro}

\begin{proof}
	Immediate from the definitions and corollary \ref{SuppCombBetti}.
\end{proof}

Going back to the lollipop, example \ref{Lollipop}, we can now compute that the first Betti-Steiner point, (only sensible to one-dimensional homological features, the two ``holes'' of the lollipop), is located at the center of the circumference, and thus, by the previous corollary, we can infer that its zeroth Betti-Steiner point is at half the length of $B$ to the right of the center of $A$. Both points now lie within the shape: these points may be better centrality measures than the usual Steiner point. As such, we propose the following definitions of Betti-Steiner supports and distances.

\begin{definition}
	Let $M$ be a compact, tame shape. Its $i$-th \textbf{Betti-Steiner} support is defined as the function
	\[ \beta_iS_M(v) := \beta_ih_M(v)-\la v, \beta_i(M)\beta_is(M) \ra\, . \]
	
	When $\beta_i(M) \neq 0$, it can equivalently defined as the Betti support of $M-\beta_is(M)$, or as
	\[ \beta_iS_M(v):= \int_{\bb R} (\beta_i(-v)(t)-\beta_i(M)\bb 1_{[-\langle v,\beta_is(M) \rangle, +\infty)}(t)) dt\, .\]
	
	The $i$-th \textbf{Betti-Steiner support norm} and \textbf{Betti-Steiner support distance}, as well as \textbf{Betti-Steiner barcodes} and \textbf{Betti-Steiner bottleneck distances} are defined analogously to definitions \ref{DefNormSupp}, \ref{DefDistSupp} and \ref{DefBarcodes} by replacing $S_M$ with $\beta_iS_M$.
\end{definition}

\begin{prop}
	\label{SumBettiSteiner}
	Let $M$ be a compact, tame shape. Then:
	\[ \sum_i (-1)^i \beta_iS_M(v) = S_M(v).\]
\end{prop}

\begin{proof}
	Recall that $S_M(v)=h_M(v)-\la v, \chi(M)s(M) \ra$. Using the results of corollaries \ref{SuppCombBetti} and \ref{StPointCombBetti} and collecting terms, this in turn equals $\sum_i  (-1)^i(\beta_ih_M -\la v, \beta_i(M)\beta_is(M)\ra) = \sum_i (-1)^i \beta_iS_M(v)$.
\end{proof}

Similar definitions can be made in the exponential light (see definition \ref{DefExpSupp}), with the added benefit that a positive function parameter is no longer needed inside the integral to ensure the integrand is positive, but with the caveat that the integral may now be zero if $\beta_i(M) = 0$, causing the exponential support to potentially evaluate as $-\infty$. All the isometry invariance properties of Steiner supports and barcodes extend to the Betti-Steiner case. We mention below the Betti-Steiner analog of the robustness under point cloud noise proven in theorem \ref{ThmPointCloud}.

\begin{theorem}[Betti support robustness under point cloud noise]
	Let $M$ be a compact, tame shape, and let $P$ be a point cloud disjoint from $M$. Then, for all $i \geq 0$:
	\[\beta_iS_{M \sqcup P} = \beta_iS_M\, .\]
\end{theorem}

\begin{proof}
	For $i>0$, the statement is trivial since as a point cloud has no positive-dimensional homological features, so all sublevel sets of $M$ and $M \sqcup P$ have the same $i$-th dimensional homology. The case $i=0$ then follows from proposition \ref{SumBettiSteiner}, allowing us to express $\beta_0(M)$ in terms of higher dimensional Betti-Steiner supports and the Steiner support of $M$, and theorem \ref{ThmPointCloud}.
\end{proof}

To end our theoretic discussion, we should point out that as Betti numbers are not additive, neither are Betti-Steiner points and supports. However, thanks to the Mayer-Vietoris exact sequence, there are some relationships to be found. As an example, here is a case where additivity can be assured:

\begin{prop}
	Let $M$ and $N$ be compact and tame. Assume that, for fixed $v \in \bb S^{n-1}$ and $i \in \bb N \cup \{0\}$, the following hold:
	\begin{enumerate}
		\item $\beta_{i+1}((M\cup N)_{-v,t})= 0$ for all $t \in \bb R$.
		\item $\beta_{i-1}((M \cap N)_{-v,t})= 0$ for all $t \in \bb R$.
	\end{enumerate}
	Then,
	\[ \beta_ih_M(v)+\beta_ih_N(v) = \beta_ih_{M \cup N}(v) +\beta_ih_{M \cap N}(v)\, . \]
	Moreover, if the above conditions hold for any $v \in \bb S^{n-1}$, then
	\[ \beta_i(M)\beta_is(M)+\beta_i(N)\beta_is(N)=\beta_i(M \cup N)\beta_is(M \cup N)+\beta_i(M \cap N)\beta_is(M \cap N)\, , \]
	and
	\[ \beta_iS_M+\beta_iS_N= \beta_iS_{M \cup N} +\beta_iS_{M \cap N}\, . \]
\end{prop}

\begin{proof}
	Let $t \in \bb R$. The Mayer-Vietoris exact sequence for the pair $M_{-v,t}$, $N_{-v,t}$ at stage $i$ reads:
	\[ 0 \to H_i((M \cap N)_{-v,t}) \to H_i(M_{-v,t}) \oplus H_i(N_{-v,t}) \to H_i((M \cup N)_{-v,t}) \to 0\, ,\]
	where the zero on the left is $H_{i+1}((M\cup N)_{-v,t})$ and the zero on the right is $H_{i-1}((M \cap N)_{-v,t})$. This short exact sequence guarantees that $\beta_i(M_{-v,t})+\beta_i(N_{-v,t})=\beta_i((M \cup N)_{-v,t})+\beta_i((M \cap N)_{-v,t})$, that is, Betti curves on direction $-v$ are additive. This immediately implies the first statement.
	
	If this holds for every $v \in \bb S^{n-1}$, then the unweighted Betti-Steiner points will obviously be additive, and so will the Betti-Steiner supports.
\end{proof}

\paragraph{Experimental results for Betti-Steiner supports}

We believe that Betti-Steiner supports and barcodes can provide more discriminative features, albeit at a higher computational cost. To assess this, we repeat the MPEG-7 dataset experiment of section \ref{SectExperiment} with Betti-Steiner supports and barcodes (of dimensions 0 and 1, as these are the only non-trivial homological features in the plane). Instead of implementing a more general persistent homology-based pipeline, we have opted for a quicker method that is specific to our contour dataset. To justify its correctness, we briefly mention this lemma:

\begin{lemma}
	\label{LemmaSubOfCircle}
	Any proper sublevel set $N$ of a shape $M$ homeomorphic to $\bb S^1$ verifies $H_1(N) = 0$.
\end{lemma}

\begin{proof}
	Pick $p\in M \setminus N$, so that $N \subseteq M \setminus \{p\} \cong \bb R$. Thus, $N$ is homeomorphic to a subset of $\bb R$ and so has trivial one-dimensional homology.
\end{proof}

As a consequence, if $M$ is a disjoint union of shapes $\{M^i\}_{i =1}^k$ with each homeomorphic to a circle, $\beta_1(M_{v,t})= \sum_{i=1}^k\beta_1(M^i_{v,t})$ is the number of $M^i$ which are completely contained in the halfspace with outer normal vector $v$ at height $t$. Thus, the only places where the Betti curve in direction $v$ will change are at the maximum heights of each component $M^i$ when viewed from direction $v$.

The procedure we follow to compute $\beta_0S_M(v)$ and $\beta_1S_M(v)$ for a collection of closed polygonal curves $M=\sqcup_i M^i$ is as follows:
\begin{itemize}
	\item We compute $\ECT(M)$ and $h_M$ (via the \texttt{ect} package and integration, as in section \ref{SectExperiment}).
	\item By lemma \ref{LemmaSubOfCircle} and the discussion thereafter, we can obtain the changes in the one-dimensional Betti curve by computing the last point hit in each curve, that is, the Betti curve in direction $v$ will increase by one at points $\max_{x \in M^i} \la x,v \ra$ for each $i=1,...,k$. We can then compute the first Betti support of $M$, $\beta_1h_M$.
	\item By corollary \ref{SuppCombBetti}, we have $\beta_0h_M = h_M +\beta_1h_M$, allowing us to compute $\beta_0h_M$.
	\item Both Betti-Steiner points are computed from this, which in turn allow us to compute Betti-Steiner supports $\beta_0S_M$ and $\beta_1S_M$.
\end{itemize}

 Vectorization of supports and barcodes, as well as training conditions, are done in the same way as in section \ref{SectExperiment}. The only difference is the function parameter $f$ chosen for the computation of exponential Betti-Steiner supports, for which we chose the identity since Betti curves are already everywhere non negative. Feature dimensions are doubled since we are now computing two values for each direction. We report average accuracies in table \ref{AccBetti}. As expected, there is an improvement over our results with Steiner supports, which is particularly noticeable in the 70 classes test, where by employing both the usual and exponential Betti-Steiner supports we achieve similar results to training over raw ECT data. Exponential Betti-Steiner supports seem to be less accurate than usual Betti-Steiner supports in the 70 class test.

\begin{table}[]
	\centering
	\begin{tabular}{l|l|l|l|l|l|}
		\cline{2-6}
		& \begin{tabular}[c]{@{}l@{}}Feature\\ dimension\end{tabular} & 10 classes                 & 70 classes        & \begin{tabular}[c]{@{}l@{}}10 classes\\ (misaligned)\end{tabular} & \begin{tabular}[c]{@{}l@{}}70 classes\\ (misaligned)\end{tabular} \\ \hline
		\multicolumn{1}{|l|}{B-S Supports}                   & 200                                                         & 98.24 \%                   & 79.86 \%          & 76.93 \%                                                          & 32.79 \%                                                          \\ \hline
		\multicolumn{1}{|l|}{B-S Landscapes}                 & 600                                                         & 94.18 \%                   & 54.71 \%          & 94.50 \%                                                          & 54.00 \%                                                          \\ \hline
		\multicolumn{1}{|l|}{B-S Supports + Landscapes}      & 800                                                         & 98.39 \%                   & 81.71 \%          & 90.10 \%                                                          & 56.79 \%                                                          \\ \hhline{|=|=|=|=|=|=|}
		\multicolumn{1}{|l|}{Exp. B-S Supports}              & 200                                                         & 98.03 \%                   & 70.21 \%          & 74.23 \%                                                          & 29.21 \%                                                          \\ \hline
		\multicolumn{1}{|l|}{Exp. B-S Landscapes}            & 600                                                         & 95.83 \%                   & 49.79  \%         & 95.54 \%                                                          & 49.14 \%                                                          \\ \hline
		\multicolumn{1}{|l|}{Exp. B-S Supp. + Landscapes}    & 800                                                         & 98.59 \%                   & 76.50 \%          & 94.24 \%                                                          & 56.43 \%                                                          \\ \hhline{|=|=|=|=|=|=|}
		\multicolumn{1}{|l|}{Both B-S Supports}              & 400                                                         & \textit{\textbf{98.78 \%}} & 83.57 \%          & 80.36 \%                                                          & 41.14 \%                                                          \\ \hline
		\multicolumn{1}{|l|}{Both B-S Landscapes}            & 1,200                                                       & 96.34 \%                   & 66.57  \%         & \textit{\textbf{96.31 \%}}                                        & 65.43 \%                                                          \\ \hline
		\multicolumn{1}{|l|}{Both B-S Supports + Landscapes} & 1,600                                                       & 98.76 \%                   & \textit{85.14 \%} & 93.54 \%                                                          & \textit{\textbf{66.57 \%}}                                        \\ \hhline{|=|=|=|=|=|=|}
		\multicolumn{1}{|l|}{Raw ECT data}                   & 10,000                                                      & 98.69 \%                   & \textbf{85.86 \%} & 79.06 \%                                                          & 45.86 \%                                                          \\ \hline
		\multicolumn{1}{|l|}{Raw ECT data (reduced)}         & 800                                                         & 98.50 \%                   & 84.79 \%          & 77.58 \%                                                          & 43.36 \%                                                          \\ \hline
	\end{tabular}
	\caption{Average accuracies of SVMs trained with Betti-Steiner supports and landscapes data. ECT accuracies are repeated from table \ref{AccTable}.}
	\label{AccBetti}
\end{table}

	\section{Conclusions and further directions of study}
\label{SectConclusion}
The simple connection between the ECT and the support function given in definition \ref{DefSupport} has proven to be a source of inspiration for defining new pseudodistances, a powerful tool to relate the ECT with geometric quantities, and provides isometry invariant features in the form of Steiner barcodes. Our early experiments indicate that Steiner supports can perform adequately in classification tasks with small, aligned datasets, and are unaffected by point cloud noise. Steiner barcodes, although less informative, are applicable in misaligned datasets and may be of use in combination with other better performing methods. Lastly, viewing the support function as a summary of homological features, Betti supports, has provided more discriminative features. Several questions remain open: as the goal of this article was to make the connection between TDA and convex geometry apparent and arrive at a workable pipeline, we have not fully explored all the concepts within. Here are some questions which we wish to look more into in the future:

\begin{itemize}
	\item How robust are the support and Steiner centered distances under perturbations or deletions?
	\item Can we find under which conditions the support / Steiner centered pseudodistance between two shapes is zero?
	\item There is a perfect match between the Hausdorff distance and the $L^\infty$ metric for convex sets (theorem \ref{HausdorffDist}). Can we exploit this to provide bounds relating the Hausdorff, support, and Steiner centered distances for general shapes?
	\item Can other geometric quantities be explicitly related to the ECT, either in the plane or in higher dimensions?
	\item Are there flat, non-convex, shapes for which a computation of their perimeter is not obvious, but have constant support width and as such can be inferred via the generalized Barbier's theorem (see theorem \ref{BarbierGen})?
	\item Can we produce more informative isometry invariant features than Steiner barcodes?
	\item Can we implement Betti-Steiner supports and barcodes in a computationally efficient way in more general settings?
	\item To what extent is the assignment $M \mapsto (\beta_ih_M)_{i \geq 0}$ non-injective?
	\item Much of the ECT theory has been extended to weighted simplicial complexes, and in particular for grayscale image analysis. Many of our constructions admit a trivial extension to this context, however, early experiments show that support functions are much less informative for pixelated images: this seems to be a clear artifact of the ``jaggedness'' of their edges. Can we extend these concepts for weighted, cubical complexes in a compelling way?
\end{itemize}

\section*{Acknowledgments and availability of code}
This work was supported by a grant issued by Spain's Ministerio de Ciencia, Innovación y Universidades (MICIU) [FPU24/03165]. The author thanks his PhD advisors Álvaro Lozano and Miguel Ángel Marco for their continued support, review and helpful comments, as well as David Alonso and Jorge Santiago Ibáñez for providing insight and resources regarding convex geometry. Finally, the author would like to thank his masters' thesis advisor, Antonio Félix Costa, for his encouragement to follow up on some early results obtained during the writing of said thesis.

Code employed for computations and SVM training can be found at the following \texttt{github} repository \cite{Gacias_Franco_ECT_Support}.

	\newpage
	\bibliography{biblioECTSupport}{}
	\bibliographystyle{unsrt}
	
	\newpage
	\appendix
	\section{Approximation of a flat tame set by simplicial complexes}

In this section, we prove the following:
\begin{theorem}
	\label{Approx}
	Let $M \subseteq \bb R^2$ be a compact and tame shape. There exists a sequence of simplicial complexes $\{S_n\}_{n \in \bb N}$ such that:
	
	\begin{enumerate}
		\item $S_n$ is homeomorphic to $M$ for all $n$.
		\item $\Peri(S_n) \overset{n \to \infty}{\longrightarrow} \Peri(M)$.
		\item $\int_{\bb S^{1}}h_{S_n}(v) dv \overset{n \to \infty}{\longrightarrow} \int_{\bb S^{1}} h_M(v) dv$.
	\end{enumerate}
\end{theorem}
In combination with the perimeter recovery formula for simplicial complexes --stating that the sequences of points two and three above are in fact equal--, this allows us to finish the proof of theorem \ref{PeriRecovery}. Before constructing $\{S_n\}_{n \in \bb N}$ in general, we do so for convex curves (recall that these are curves such that the line joining any two of its points lies \emph{above} it):
\begin{lemma}
	\label{ConvexCurve}
	Let $C \subseteq \bb R^2$ be a convex, compact, tame, and $\mathcal C^1$ curve joining two horizontally aligned points. Then, a sequence $\{S_n\}_{n \in \bb N}$ as in theorem \ref{Approx} exists. In particular, the perimeter recovery formula holds for $C$.
\end{lemma}
\begin{proof}
	Since $C$ is $\mathcal C^1$, one can pick evenly spaced points on $C$ to obtain approximations of it by polygonal curves whose lengths approach the length of $C$. This sequence obviously verifies the first two required properties. To prove the third, write $C= C_{v,t} \cup C_{-v,-t}$ and apply additivity to get \[\ECT(C)(v,t)+\ECT(C)(-v,-t)=1+\chi(C \cap L_{v,t})\, ,\]
	where $L_{v,t}$ is the line segment with normal vector $v$ placed at height $t$. Splitting the integral over the upper and lower hemispheres, and noting that $1=\bb 1_{[0, +\infty)}(t)+\bb 1_{(0, +\infty)}(-t)$ we find
	\[ \int_{\bb S^{1}} h_C(v) = 2 \int_{\bb S^{1}_+} \int_{\bb R} \chi(C \cap L_{v,t}) dt dv\, . \]
	The same holds when replacing $C$ by $S_n$. Let $p_n: (v, t) \mapsto \chi(S_n \cap L_{v,t})$, and $p: (v,t) \mapsto \chi(C \cap L_{v,t})$. Since $S_n$ and $C$ are convex curves, $p_n$ and $p$ cannot surpass $2$. Thus, every function $p_n$ is bounded by the constant function with value $2$ on the compact subset of $\bb S^1_+ \times \bb R$ representing the lines which intersect $C$. By the dominated convergence theorem, we need only prove that $p_n$ converges to $p$ almost everywhere, that is, for almost every line $\ell$, the number of points of $S_n$ in it is eventually equal to the number of points of $C$ in it. We can ignore lines which intersect $S_n$ or $C$ at intervals, and lines tangent to $C$, since they have measure zero. Thus, we distinguish three cases:
	\begin{enumerate}
		\item $p(v,t)=0$: $\ell$ is disjoint from $C$. Then it is also disjoint from $S_n$, as $S_n$ lies in the convex closure of $C$. Thus $p_n(v,t)=0$.
		\item $p(v,t)=1$: $\ell$ intersects $C$ at one point, non tangentially. A Jordan curve-type argument applied to the curves drawn by the edges of $S_n$ and the arcs of $C$ separated the vertices of $S_n$ shows that $\ell$ must touch $S_n$ at exactly one point (if it crossed another, it would inevitably hit $C$ again). Thus $p_n(v,t)=1$.
		\item $p(v,t)=2$: $\ell$ intersects $C$ at two points, $p_1$ and $p_2$. For large enough $n$, $p_1$ and $p_2$ will lie in two different arcs separated by the vertices of $S_n$. The same argument as before shows that $\ell$ must hit $S_n$ twice. Thus $p_n(v,t) \overset{n \to \infty}{\longrightarrow} 2$.\qedhere
	\end{enumerate}
\end{proof}

\begin{proof}[Proof of theorem \ref{Approx}]
	We first describe a triangulation of $\partial M$ (that is, the image of boundary and isolated edges of any triangulation of $M$). Since $\partial M$ consists of finitely many $\mathcal C^1$ curves, consider for each an approximation by polygonal curves whose vertices lie in $\partial M$. Choose these decompositions fine enough so that for each line segment of the polygonal curves, the arc of $\partial M$ drawn between the two points it joins is either concave, convex, or perfectly matches the line segment. Naming $S$ this set of polygonal curves, let $f: S \to \partial M$ be this $\mathcal C^1$ triangulation. Note that, by construction, all vertices of $S$ are fixed points of $f$. We now extend $f$ to a triangulation $\tilde f$ of $M$, as follows:
	\begin{itemize}
		\item Triangulate the convex closure of $S$ in a compatible way with the natural triangulation of $S$. By applying barycentric subdivision if necessary, suppose each added face contains at most one edge of $S$ on its boundary.
		\item Fix $\tilde f(v) = v$ for all vertices $v$ of the resulting triangulation, and $\tilde f|_{S}=f$. Extend $\tilde f$ linearly to all added edges, and then to all faces (non-linearly on faces with a boundary edge).
		\item The interior of each $2$-simplex now either maps completely inside $M$ or outside $M$. Only add to $S$ all those who map to $M$, together with their boundaries.
		\item Finally, add any isolated points $M$ might have to $S$ and extend $\tilde f$ as the identity on them.
	\end{itemize}
	
	By repeating this process with finer and finer triangulations of $\partial M$, we obtain a sequence $S_n$ of simplicial complexes who, by construction, verify the first two points of the theorem. As for the third, let us show that
	\[ \int_{\bb S^{1}} (h_{S_n}(v)-h_{M}(v))  dv \overset{n \to \infty} \longrightarrow 0\, .\]
	Using additivity, decompose $S_n$ as the union of its faces, isolated edges, and isolated points, and decompose $M$ as the image under $\tilde f$ of that same union. By construction of $\tilde f$, faces that don't contain boundary edges, internal edges and all vertices remain fixed under $\tilde f$. These terms cancel out, so the integral equals
	\[ \sum_{F} \int(h_F -h_{\tilde f(F)}) + \sum_{e} \int(h_e -h_{\tilde f(e)})\, , \]
	where the first sum ranges over faces of $S_n$ with a boundary edge, and the second over isolated edges. Focusing first on the isolated edges, $\int h_e$ is twice the length of the straight-line segment $e$, and $h_{\tilde f(e)}$ is twice the length of the arc of $M$ between the endpoints of $e$, by lemma \ref{ConvexCurve}.
	
	For faces, we distinguish between two cases: for a face $F$, $\tilde f(F)$ differs from $F$ by either introducing a ``bump'' or a ``dent'' (see figure \ref{BumpDent}), depending on whether the image under $\tilde f $ of its boundary edge is concave or convex.
	\begin{figure}
		\centering
	\includegraphics[scale = 0.45]{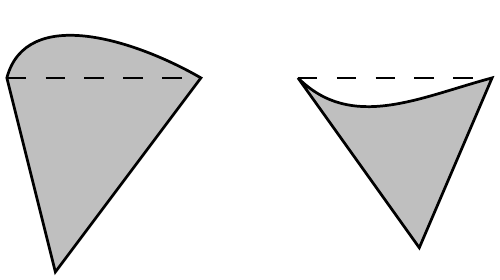}
	\caption{Two faces of the triangulation $S_n$ and their images under $\tilde f$. The first is a bump, the second is a dent. In each case, the part added or removed to pass from $F$ to $\tilde f(F)$ is convex.}
	\label{BumpDent}
	\end{figure}
	In the case of a bump, we have $\tilde f(F)= F \cup B$, with $B$ being the added, convex, bump. Thus, by additivity $h_F-h_{\tilde f(F)}= h_{F \cap B}-h_B$. Now, $F \cap B$ is nothing but the boundary edge of $F$, so both terms on the right are convex. Since the perimeter formula holds for convex shapes\footnote{If one wishes to not use this fact, a similar proof to the one of lemma \ref{ConvexCurve} shows the assertion for bumps.}, the difference of their integrals turns out to be the length of the boundary edge minus the length of its image under $\tilde f$.
	
	In the case of dent, write $F= \tilde f(F) \cup D$, with $D= \overline{F \setminus \tilde f(F)}$ being the dent, and applying additivity, we have $f_F-h_{\tilde f(F)}= h_D-h_{D \cap \tilde f(F)}$. $D$ is again convex, with one component of its boundary being the boundary edge of $F$, and the other its image under $\tilde f$. On the other hand, $D \cap \tilde f(F)$ is a convex curve, so by applying lemma \ref{ConvexCurve}, the difference of integrals again turns out to be the length of the boundary edge minus the length of its image under $\tilde f$.
	
	Collecting terms, the difference of integrals for any fixed $n$ equals:
	\[\Peri(S_n)-\Peri(M). \]
	Since $S_n$ verifies point two, this converges to zero.
	\end{proof}
	
\end{document}